\documentclass[10pt,a4paper,final]{article}
\usepackage[latin1]{inputenc}
\usepackage{amsmath}
\usepackage{amsfonts}
\usepackage{amssymb}
\usepackage{makeidx}
\usepackage{graphicx}
\usepackage{hyperref}

\usepackage{a4wide}

\usepackage{todo}
\usepackage{amssymb}
\usepackage{amsmath}
\usepackage{latexsym}
\usepackage{amssymb}
\usepackage{amsthm}
\usepackage{cite}
\usepackage{mathrsfs}
\usepackage{bbm}
\usepackage{bbold}
\usepackage{subfigure}
\usepackage{float}
\usepackage{epsfig}
\usepackage[numbers]{natbib}
\usepackage{mathtools}
\usepackage{listings}
\usepackage{pgf}
\usepackage{listings}
\usepackage{epstopdf}

\theoremstyle{plain} 
\newtheorem{thm}{Theorem}[section] 
 
\newtheorem{prop}[thm]{Proposition}

\newtheorem{lem}[thm]{Lemma}

\newtheorem{fact}[thm]{Fact}

\theoremstyle{remark}

\title{Dispersion on the Complete Graph\thanks{An extended abstract containing some results of this work appeared in the proceedings of EUROCOMB '23~\cite{DMP23proc}.}}

\author{Umberto ~De Ambroggio\thanks{University of Munich, Department of Mathematics, Theresienstr.~39, 80333 Munich, Germany. Email: \texttt{\{deambrog,makai,kpanagio\}@math.lmu.de}. Supported by ERC Grant Agreement 772606-PTRCSP.}
\and Tam\'as ~Makai\footnotemark[2]
\and Konstantinos ~Panagiotou\footnotemark[2]
}

\begin{document}
	
\maketitle

\begin{abstract}
We consider a synchronous process of particles moving on the vertices of a graph~$G$, introduced by Cooper, McDowell, Radzik, Rivera and Shiraga (2018). Initially,~$M$ particles are placed on a vertex of~$G$. At the beginning of each time step, for every vertex inhabited by at least two particles, each of these particles moves independently to a neighbour chosen uniformly at random. The process ends at the first step when no vertex is inhabited by more than one particle. 
		
Cooper et al.~showed that when the underlying graph is the complete graph on~$n$ vertices, then there is a phase transition when the number of particles $M = n/2$. They showed that if $M<(1-\varepsilon)n/2$ for some fixed $\varepsilon>0$, then the process finishes in a logarithmic number of steps, while if $M>(1+\varepsilon)n/2$, an exponential number of steps are required with high probability. Here we provide a thorough asymptotic analysis of the dispersion time around criticality, where $\varepsilon = o(1)$, and describe the transition from logarithmic to exponential time. As a consequence of our results we establish, for example, that the dispersion time is in probability and in expectation in $\Theta(n^{1/2})$ when $|\varepsilon| = O(n^{-1/2})$, and provide qualitative bounds for its tail behavior.
\end{abstract}
	
\section{Introduction}
We consider the synchronous \textit{dispersion process} introduced by Cooper, McDowell, Radzik, Rivera and Shiraga \cite{CMRRS18}.
It involves particles that move between vertices of a given graph $G$. A particle is called \emph{happy} if there are no other particles on the same vertex, otherwise it is \emph{unhappy}. Initially,~$M \ge 2$ particles are placed on some vertex of $G$ and are all unhappy. In every (discrete) time step, all unhappy particles move simultaneously and independently to a neighbouring vertex selected uniformly at random. Happy particles do not move. The process terminates at the first time at which all particles are happy; this (random) time is called the \textit{dispersion time} and constitutes the main object of interest here. 

In \cite{CMRRS18} the authors studied this process on several graphs, and established results on the dispersion time and the maximum distance of any particle from the origin at the point in time of dispersion.
A particular case considered in~\cite{CMRRS18} is the behaviour  when the underlying graph is the complete graph $K_n$ with $n$ vertices. In that case we denote the dispersion time by $T_{n,M}$, since, due to symmetry, the choice of the initial vertex does not matter. The most general results come from considering a \textit{lazy} variant of the dispersion process, which was shown to disperse more particles in a comparable number of steps. More precisely, in this lazy version any unhappy particle moves with probability $q\in (0,1]$ and stays at its current location with probability $1-q$.
The main result of~\cite{CMRRS18} regarding the dispersion time of this variant of the process, which we denote by  $T_{n,M}(q)$, is that there are constants $c,C > 0$ such that, if $M=(1-q/2-\alpha)n$ where $\alpha > 0$ is allowed to depend on $n$, then
\begin{equation}\label{clazy0}
    T_{n,M}(q) \le  C (q\alpha)^{-1}\ln n ~\text{ with probability at least }~ 1-C/n,
\end{equation}
whereas when $M=n(1-q/2+\alpha)$, then 
\begin{equation}\label{clazy}
    T_{n,M}(q) \geq e^{c n q^3 \alpha^4} ~ \text{ with probability at least } ~ 1-e^{-c n q^3 \alpha^4}.
\end{equation}
The above statements leave several questions open. Indeed, corresponding bounds for the lower and upper tails of $T_{n,M}$ were not provided. 
Moreover, it is not clear  what the actual behavior is when $M$ is close to $n/2$, that is, when $M = (1+\varepsilon)n/2$ for some $|\varepsilon| = o(1)$ and how the transition from logarithmic to exponential time quantitatively looks like.  
For example,~\eqref{clazy} is not informative when $\alpha=o(n^{-1/4})$, as it essentially only states that the number of steps is at least one. 

Our main contribution is a thorough  analysis of the dispersion process when $q = 1$ and, as above, we assume that $M=(1+\varepsilon)n/2$ and $|\varepsilon| = o(1)$. We establish that the process exhibits three qualitatively different behaviours based on the asymptotics of $\varepsilon$, where, informally speaking, $T_{n,M}$ smoothly changes from $|\varepsilon|^{-1} \ln(\varepsilon^2n)$ to $n^{1/2}$ and then to $\varepsilon^{-1} e^{\Theta(\varepsilon^2 n)}$; in particular, $T_{n,M} = \Theta(n^{1/2})$ whenever $M = n/2 + O(\sqrt{n})$ within a \emph{scaling window} of size $O(\sqrt{n})$.
We continue with a detailed and formal description of the dispersion process and state  our results.

\paragraph{Model \& Results.} Let $n \in \mathbb{N}$. We denote by $M\ge 2$  the number of particles in the dispersion process, and we write $\mathcal{U}_t$ and $\mathcal{H}_t$ for the \textit{sets} of unhappy and happy particles at time $t$, respectively.
Moreover, we set $U_t\coloneqq |\mathcal{U}_t|$ and $H_t\coloneqq |\mathcal{H}_t|$.
The process evolves as follows. Initially, which is at time $t=0$, all $M$ particles are placed on \textit{one} distinguished vertex, say vertex 1, and are unhappy. Thus, writing $p_i$ for the $i$-th particle, $1\le i \le M$,  we set
\[
    \mathcal{U}_0 = {\cal P}\coloneqq \{p_1, \dots, p_M\},~~ U_0 = M
    \quad \text{and} \quad 
    \mathcal{H}_0 = \emptyset,~~ H_0 = 0.
\]
For every $t\in \mathbb{N}_0$, the distribution of $\mathcal{H}_{t+1}$ (and thus also of $\mathcal{U}_{t+1}, U_{t+1},H_{t+1}$), given $\mathcal{U}_{t}$, is defined as follows.
Each particle in $\mathcal{U}_{t}$ moves to one of the $n$ vertices selected independently and uniformly at random and each particle in $\mathcal{H}_t$ remains at its position. 
In particular, if we denote by $p_{i,t}$ the position of particle $i$ at time $t\in \mathbb{N}_0$, then we set
\[
    p_{i,0} = 1, \quad 1\le i \le M,
\]
and, in distribution, 
\[
    p_{i,t+1} =
    \begin{cases}
        p_{i,t}, & \text{ if } p_i \in {\cal H}_t, \\
        G_{i,t+1}, & \text{ if } p_i \in {\cal U}_t, 
    \end{cases},
    \quad 1 \le i \le M, t \in \mathbb{N}_0,
\]
where $(G_{i,t})_{1 \le i \le n, t\in\mathbb{N}_0}$ are independent and uniform from $\{1, \dots, n\}$. In addition, we set for $t \in \mathbb{N}_0$
\[
    {\cal H}_{t+1} = \big\{ p_i \in {\cal P}: p_{i,t+1} \neq p_{j,t+1} \text{ for all $j \in \{1, \dots, M\} \setminus \{i\}$}\big\},
    \quad
    {\cal U}_{t+1} = \{p_1, \dots, p_M\} \setminus {\cal H}_{t+1}.
\]
With this notation, the \emph{dispersion time}, which is our main object of study, is defined as the smallest $t$ at which there are no unhappy particles, that is,
\begin{equation}
\label{eq:TnM}
    T_{n, M} \coloneqq \inf\{t \in \mathbb{N}_0 : U_t = 0\}.
\end{equation}
Our main results are the following two theorems. The first one addresses the upper tail of $T_{n,M}$.
\begin{thm}\label{mainthm}
There is a $C > 0$ such that the following is true for sufficiently large $n$ and all $A \ge 1$. Let $M = (1+\varepsilon)n/2 \in \mathbb{N}$, where $\varepsilon = o(1)$  and $|\varepsilon| < 1/9$. If $\varepsilon< -e n^{-1/2}$, then
	$$\mathbb{P}\big(T_{n,M}>A C |\varepsilon|^{-1}\ln(\varepsilon^2 n)\big)\le e^{-A}.$$
	Moreover, if $|\varepsilon| \le e n^{-1/2}$, then
	$$\mathbb{P}\big(T_{n,M}>A C n^{1/2}\big)\le e^{-A}.$$
	Finally, if $\varepsilon > e n^{-1/2}$, then
	$$\mathbb{P}\big(T_{n,M}> A \varepsilon^{-1} e^{C \varepsilon^2n}\big)\le e^{-A}.$$
\end{thm}
\noindent
Our second main result provides bounds for the lower tail of the distribution of $T_{n,M}$.
\begin{thm}\label{mainthm2}
    There is a $c > 0$ such that the following is true for sufficiently large $n$ and all $A \ge 1$. 
	Let $M = (1+\varepsilon)n/2 \in \mathbb{N}$, where $|\varepsilon| = o(1)$ and $|\varepsilon| < 1/9$. If $\varepsilon< -e n^{-1/2}$, then
	$$\mathbb{P}\big(T_{n,M}\le c |\varepsilon|^{-1}\ln(\varepsilon^2 n)/A\big)\le e^{-A}.$$
	Moreover, if $|\varepsilon| \le e n^{-1/2}$, then
	$$\mathbb{P}\big(T_{n,M}\le c n^{1/2}/A\big)\le e^{-A}.$$
	Finally, if $\varepsilon > e n^{-1/2}$, set $k_0\coloneqq e^{c\varepsilon^2 n}$. Then
	$$\mathbb{P}\big(T_{n,M}< \varepsilon^{-1}k_0/A\big)\le
\left\{
\begin{array}{ll}
      \exp\left(-\frac{A \varepsilon^2 n}{k_0 }\right) & \mbox{, if } A>k_0 \\
      A^{-1} &\mbox{, if } A\le k_0
\end{array}
\right. \enspace .
$$
\end{thm}
\noindent
Let us briefly discuss some consequences of our results. First of all, the two theorems combined imply that in probability
\[
    T_{n,M} = \Theta(|\varepsilon|^{-1} \ln(\varepsilon^2n))
    \quad \text{if} \quad \varepsilon < -en^{-1/2},
\]
and 
\[
    T_{n,M} = \Theta(n^{1/2})
    \quad \text{if} \quad |\varepsilon| = O(n^{-1/2}).
\]
For larger $\varepsilon$ we obtain the slightly weaker uniform estimate that in probability 
\[ 
    \ln(T_{n,M})
    = \Theta( \varepsilon^2 n +  \ln( \varepsilon^{-1}))
    \quad \text{if} \quad \varepsilon = \omega(n^{-1/2}).
\]
These results imply that there is critical window around $n/2$ of order $n^{1/2}$, where the dispersion time is of order $n^{1/2}$ as well. Moreover, the dispersion time increases smoothly from logarithmic to $n^{1/2}$ and then to exponential when we gradually increase~$\varepsilon$.
Apart from these estimates we can also use our main theorems to obtain information about, for example, the expectation of $T_{n,M}$. In particular, Theorem~\ref{mainthm} guarantees that $T_{n,M}$ has  an exponential(-ly thin) upper tail and so~$T_{n,M}$ is, after appropriate normalization, integrable; we readily  obtain 
\[
    \mathbb{E}[T_{n,M}] = \Theta(|\varepsilon|^{-1} \ln(\varepsilon^2n))~ \text{  if  } ~ \varepsilon \le -en^{-1/2},
    \qquad
    \mathbb{E}[T_{n,M}] = \Theta(n^{1/2}) ~\text{ if }~ |\varepsilon|  = O(n^{-1/2}),
\]
and
\[
    \ln \mathbb{E}[T_{n,M}] = \Theta(\varepsilon^2 n +  \ln(\varepsilon^{-1})) ~\text{ if }~ \varepsilon = \omega(n^{-1/2}).\\
\]

\paragraph{Directions for Future Research.} It seems plausible to think that the dispersion time is \textit{monotone increasing} with respect to the number of particles, in the following sense. Let $M\leq M'$. Then we conjecture that
\begin{equation}\label{conjecture}
    \mathbb{P}(T_{n,M} \ge t) \leq \mathbb{P}(T_{n,M'} \ge t), \quad t\in\mathbb{N}_0,
\end{equation}
that is, $T_{n,M'}$ stochastically dominates $T_{n,M}$. Such a statement would have helped us at many places in our proofs. However, proving anything that even comes close to~\eqref{conjecture} seems out of reach.

Another important research direction is to study the distribution of $T_{n,M}$ in more detail. For example, we have established that $T_{n,M} = O_p(\sqrt{n})$ and $\mathbb{E}[T_{n,M}] = \Theta(\sqrt{n})$ if $M = n/2 + O(\sqrt{n})$ is within the critical window of the process. In the follow-up paper~\cite{ar:aofa} we prove that after the obvious rescaling,~$T_{n,M}$ converges in distribution to a non-trivial random variable, that is, for $\alpha \in \mathbb{R}$ and as $n\to\infty$
\[
    \frac{T_{n, n/2 + \alpha \sqrt{n} + o(\sqrt{n})}}{\sqrt{n}} \stackrel{d}\to X_\alpha,
    \quad
    \text{where $X_\alpha \ge 0$ is absolutely continuous}.  
\]
On the other hand, if $M$ is not in the critical window, that is, $|M-n/2| = \omega(\sqrt{n})$, then studying convergence properties of $T_{n,M}$ is an important open problem. 

\paragraph{Related Work.} The dispersion process was also studied by Frieze and Pegden~\cite{FP18}, who sharpened the result on the dispersion distance on the infinite path that was established in \cite{CMRRS18}. In particular, it was shown in~\cite{CMRRS18} that with high probability, the dispersion distance on the infinite path for $n$ particles is between $n/2$ and $O(n\ln n)$. Subsequently, in~\cite{FP18} the logarithmic factor in the upper bound was eliminated. A similar setup  was considered by Shang~\cite{S20}, where the author studied the dispersion distance on the infinite path in a non-uniform dispersion process in which an unhappy particle moves in the following step to the right with probability $p_n$ and to the left with probability $1-p_n$, independently of the other particles.

Processes where particles move between the vertices of a graph have been widely studied over the past decades; we refer the reader to \cite{CMRRS18} for references. Concerning processes whose scope is to \textit{disperse} particles on a discrete structure, a well-studied  model is Internal Diffusion Limited Aggregation; see \cite{DF91,LBG92}. In this model, particles sequentially start (one at a time) from a specific vertex designated as the origin.
Each particle moves randomly until it finds an unoccupied vertex; then it occupies it forever (meaning that it does not move at subsequent steps of the process). When a particle stops, the next particle starts moving. We emphasize that whenever a particle jumps to an occupied vertex, it just keeps moving without activating the occupant particle. In the dispersion process, on the other hand, when a (happy) particle  standing alone on a node is reached by another particle, it is reactivated and keeps moving until it becomes happy again.

\subsection{Main Proof Ideas} 
A key quantity in our proofs is the \emph{drift}, that is, the expected one-step change in the number of unhappy particles. In Section~\ref{estimatecondexp} we establish that
\begin{equation*}
    \label{eq:Eapprox}
    \mathbb{E}\big[U_{t+1} -U_t \mid U_t\big]
    =
    \varepsilon U_t - \Theta\left(U_t^2/n\right),
    \quad
    t\in\mathbb{N}_0.
\end{equation*}
So, as long as $U_t$ is much larger than $|\varepsilon| n$, then $U_{t+1} - U_t$ is a large negative number in expectation. In other words,
when there are `many' unhappy particles at time $t$, the expected  number of $U_{t+1}$ decreases by a considerable amount. However, this is no longer true when $U_t = O(|\varepsilon| n)$: then the expected decrease of $U_t$ might be tiny, or when $\varepsilon>0$, we can even expect that the number of unhappy particles \emph{increases}. We will use different methods to study the trajectory of $U_t$, depending on the range of~$\varepsilon$ and whether we are considering the upper or the lower tail. 

We briefly comment on the main ideas required to prove the  bounds on the dispersion time in Theorems \ref{mainthm} and \ref{mainthm2}.
We start by discussing the upper tail.
In the main argument we first use drift analysis~\cite{KK19} that relies on the behaviour of $\mathbb{E}[U_{t+1}-U_t \mid U_t]$
to considerably reduce the number of unhappy particles.
Once $U_t$ is sufficiently small we relate it to a different process that is easier to analyse, which we call the \emph{binomial process} (not to be confused with the homonymous process from probability theory!).
The binomial process is defined by $B_0 = U_{t_0}$ for some $t_0\in\mathbb{N}$ and recursively by setting $B_{t+1}=2\,\text{Bin}(B_{t},M/n)$, $t \in \mathbb{N}_0$. It is not very hard to establish that~$B_{t}$ provides an upper bound for $U_{t_0+t}$, as the probability that an unhappy particle moves to the same vertex as some other particle is at most~$M/n$, and in that case we account for two unhappy particles; see Section~\ref{couplingfacts} for the precise statements. Observe that the binomial process has a clear interpretation in terms of branching processes: it is equivalent to $B_0$ independent copies of a Galton-Watson branching process that have no offspring with probability $1-M/n$ and two offspring with probability $M/n$. By establishing an upper bound on the number of generations of each of these branching process trees, we then obtain the desired bounds for~$U_{t_0+t}$ for all $t\in\mathbb{N}$. See Section~\ref{sec:uppertail} for the details.

The study of the lower tail of $T_{n,M}$ is more involved. The basic idea though is again to compare $(U_t)_{t\in\mathbb{N}_0}$ to a binomial process, this time using different parameters. Indeed, an effective coupling is possible as long as $U_t$ remains sufficiently small; for example, if $U_t \le \delta n$, then an argument similar to the one in the upper tail shows that $U_{t+1}$ stochastically dominates $2\, \text{Bin}(U_t,M/n - O(\delta))$.
When studying the lower tail of $U_t$ we then distinguish two regimes according to the value of $\varepsilon$.
If $\varepsilon$ is `small', then we use the binomial process to bound (from below) the number of unhappy particles, since in that case the process is quite unlikely to become larger than $\delta n$.
However, when $\varepsilon$ is 'large', we take a different route.
We consider two sequences of stopping times, corresponding to the (random) times at which $U_t$ goes below and above, respectively, two well-defined thresholds, both located at $\Theta(\delta n)$. While $U_t$ remains below the upper margin of the strip, the coupling to the binomial process is applicable. We also show that a crossing of the above mentioned strip requires approximately $|\varepsilon|^{-1}$ steps. Thus, by showing that there are \textit{at least} $e^{\Theta(\varepsilon^2 n)}$ such crossings, we immediately obtain the desired lower bound on the dispersion time. See Section~\ref{sec:lowertail} for the details.

\paragraph{Notation.} Let $\mathbb{N}$ denote the set of positive integers and set $\mathbb{N}_0=\mathbb{N}\cup \{0\}$. Given $k\in \mathbb{N}$, we write $[k]\coloneqq \{1,\dots,k\}$. 
Given functions $f:\mathbb{N}\mapsto \mathbb{R},g:\mathbb{N}\mapsto \mathbb{R}$, we write either $f\ll g$ or $f=o(g)$ if $f(n)/g(n)\rightarrow 0$ as $n\rightarrow \infty$ and $f=O(g)$ if there is a constant $C>0$ such that $f(n)\leq Cg(n)$ for all large enough $n$. We write $f=\Theta(g)$ if $f=O(g)$ and $f=\Omega(g)$, where the latter is the same as $g = O(f)$.
Given random variables $X$ and $Y$, we write $X=_dY$ if $X$ and $Y$ have the same distribution, whereas we write $X\leq_{sd} Y$ if $Y$ stochastically dominates $X$, meaning that $\mathbb{P}(X\ge x)\leq \mathbb{P}(Y\ge x)$ for all $x \in \mathbb{R}$. Moreover, given a $p \ge 0$, some events $A,B$ and a $\sigma$-algebra~$\cal F$, when we say that $\mathbb{P}(A \mid {\cal F}) \ge p$ \emph{on $B$}, then we mean that this inequality is satisfied only when $B$ occurs, that is, $\mathbb{1}_B \cdot \mathbb{P}(A \mid {\cal F}) \ge \mathbb{1}_B \cdot p$.

\section{Preliminaries}\label{prel}

In this section we recall some known results as well as to introduce some tools that will be useful in our proofs. In Section \ref{knownresults} we recall a concentration inequality and a basic result from drift analysis. Subsequently, in Section \ref{estimatecondexp} we collect some useful estimates concerning the (conditional) expected value of $U_{t+1}$ given $U_t$.
Then, in Subsection \ref{driftsubs}, we state and prove a result providing an upper bound on the number of steps required to reduce the number of unhappy particles from some number $\ell$ to some lower number  $h<\ell$.
We conclude by introducing a rather useful relation between the (conditional) distribution of $U_{t+1}$ (given $U_t$) and a branching process with offspring distribution supported on $\{0,2\}$, and we prove a simple lemma that provides tight bounds for the probability that this branching process survives for at least a given number $k\in \mathbb{N}$ generations.

\subsection{Drift Analysis, Martingale Estimates \& Bernoulli's inequalities}
\label{knownresults}

The first fact that we recall here is a result from drift analysis, see \cite{KK19} for an extensive survey.
\begin{thm}[Thm.~10 in \cite{KK19}]\label{driftthm}
Let $(X_t)_{t\in \mathbb{N}_0}$ be a sequence of random variables over $\mathbb{R}$, $x_{min}> 0$ and $T=\inf\{t:X_t< x_{min}\}$. 
Additionally, let $D$ denote any real interval that contains all values $x\ge x_{min}$ that can be taken by $X_t$ for $t\le T$. Furthermore, suppose that $X_0\ge x_{min}$ and $X_t\ge 0$ for all $t\le T$. If there exists a monotonically increasing function $f:D \mapsto \mathbb{R}^+$ such that
\begin{equation*}\label{condition}
    X_t-\mathbb{E}\big[X_{t+1} \mid X_0,\ldots,X_t\big]\geq f(X_t),
    \quad
    t < T,
\end{equation*}
then 
\begin{equation*}
    \mathbb{E}[T\mid X_0]\leq \frac{x_{min}}{f(x_{min})}+\int_{x_{min}}^{X_0}\frac{1}{f(z)}dz.
\end{equation*}
\end{thm}
As it will be quite handy later on, we recall a well-know concentration inequality for supermartingales with bounded increments.
\begin{thm}[Azuma-Hoeffding's inequality, Thm.~23.16 in~\cite{FKIRG}]
\label{azumahoeffding}
Let $(X_t)_{t\in \mathbb{N}_0}$ be a real-valued supermartingale and let $N\in \mathbb{N}$. Suppose that $|X_i-X_{i-1}|\leq c_i$ for all $1\leq i\leq N$. Then
\[
    \mathbb{P}(X_N-X_0\ge b)\leq \exp\left\{-\frac{b^2}{2\sum_{i=1}^{N}c^2_i}\right\},
    \quad
    b \ge 0 .
\]
\end{thm}
\noindent
We will also need the following simple estimates that we will refer to as 'Bernoulli's inequalities'.
\begin{thm}[Bernoulli's inequalities, (5.3.78-79) in~\cite{MR3408971}]\label{thm:Bernoulli}
For $x\in (-1,\infty)$
\begin{align*}
    (1+x)^\alpha \ge 1+\alpha x ~\mbox{ if }~ \alpha\in \mathbb{R}\setminus (0,1)
    ~~~\text{ and }~~~
    (1+x)^{\alpha}  \le 1+\alpha x ~\mbox{ if }~ \alpha \in [0,1].
\end{align*}
\end{thm}

\subsection{Bounds on $\mathbb{E}[U_{t+1} \mid U_t]$}\label{estimatecondexp}
In this section we provide bounds for the (conditional) expected value of~$U_{t+1}$, the number of unhappy particles at the end of step $t+1$, given $U_t$. Recall that we denote by $n$ the number of vertices, by $M$ the total number of particles, whereas $H_t=M-U_t$ is the number of happy particles. Moreover, ${\cal U}_t$ and ${\cal H}_t$ denote the sets of unhappy and happy particles at time $t$. Let
\[
    X_{t+1}\coloneqq \left|\mathcal{H}_{t} \cap \mathcal{U}_{t+1}\right|
    \quad \text{ and} \quad
    Y_{t+1}\coloneqq \left|\mathcal{U}_{t}\cap \mathcal{H}_{t+1}\right|,
\]
that is,  $X_{t+1}$ is the number of particles that were happy at time $t$ but become unhappy during step $t+1$ (because they were reached by some unhappy particles) and $Y_{t+1}$ is the number of unhappy particles at time $t$ that become happy at time $t+1$ (because they move to an unoccupied vertex and none of the other unhappy particles move to that vertex). Then, by definition of the process,
\[
    U_{t+1} - U_t =  X_{t+1}-Y_{t+1}.
\]
The next simple lemma determines $\mathbb{E}[U_{t+1} \mid U_t]$.
\begin{lem}\label{mainexpvalueunhappy}
    Let $t\in \mathbb{N}_0$. Then 
\begin{equation*}
\mathbb{E}\big[U_{t+1} \mid  U_t\big]
= U_t\left(1-\frac{n-H_t}{n}\left(1-\frac{1}{n}\right)^{U_t-1}\right)+H_t\left(1-\left(1-\frac{1}{n}\right)^{U_t}\right).
\end{equation*}
\end{lem}
\begin{proof}
An unhappy particle $p$ at time $t$ is happy at time $t+1$ with probability 
$\frac{n-H_t}{n}(1-1/n)^{U_t-1}$, which is the probability that $p$ jumps to one of the $n-H_t$ unoccupied vertices, say vertex $v_p$, and no other particle in $\mathcal{U}_{t}$ jumps to $v_p$. So,
\[
    \mathbb{E}\big[Y_{t+1} \mid U_t\big]
    = U_t\frac{n-H_t}{n}\left(1-\frac{1}{n}\right)^{U_t-1}.\]
Similarly,
\[
    \mathbb{E}\big[X_{t+1} \mid U_t\big]
    = H_t\left(1-\left(1-\frac{1}{n}\right)^{U_t}\right),
\]
because a happy particle $p$ at time $t$ becomes unhappy at time $t+1$ if at least one particle in 
$\mathcal{U}_{t}$ jumps onto $v_p$, 
an event occurring with probability 
$1-(1-1/n)^{U_t}$.
\end{proof}
The next two lemmas provide handy bounds for $\mathbb{E}[U_{t+1} \mid U_t]$.
\begin{lem}\label{expunhappyupper}
Let $M=(1+\varepsilon)n/2$ with $|\varepsilon| = o(1)$ and $|\varepsilon| < 1/9$. Then, for $n \ge 100$,
\begin{equation*}
    \mathbb{E}\big[U_{t+1} \mid U_t\big]\geq (1+\varepsilon)U_t-\frac{7U^2_t}{2n}
    \quad
    \text{and}
    \quad
    \mathbb{E}\big[U_{t+1} \mid U_t\big]\geq \frac{1+\varepsilon}{3}U_t,
    \qquad
    t\in\mathbb{N}_0.
\end{equation*}
\end{lem}
\begin{proof}
Recalling Lemma~\ref{mainexpvalueunhappy} and using $M = U_t + H_t$ we obtain 
\begin{align}\label{pghstar}
\mathbb{E}\big[U_{t+1} \mid U_t\big]
        \nonumber&=
        U_t\left(1-\frac{n-H_t}{n}\left(1-\frac{1}{n}\right)^{U_t-1}\right)
        + H_t\left(1-\left(1-\frac{1}{n}\right)^{U_t}\right) \\
	& = M\left[1-\left(1-\frac{1}{n}\right)^{U_t}\right]
	    + U_t\left(1-\frac{1}{n}\right)^{U_t} \frac{H_t-1}{n-1}.
	\end{align}
Suppose first that $H_t=0$ (that is, $U_t=M$). Note that, in this case, the first lower bound in the lemma follows immediately from the observation
\[
    (1+\varepsilon)U_t - \frac{7U_t^2}{2n}
= (1+\varepsilon)M-\frac{7}{4}(1+\varepsilon)M
    < 0.
\]
In order to prove the second inequality, note that the expression in~\eqref{pghstar} equals
\begin{equation*}
    M\left[1-\left(1-\frac{1}{n}\right)^{M}\right]
	    -\frac{M}{n-1}\left(1-\frac{1}{n}\right)^{M}.
\end{equation*}
Using $(1-1/n)^M \le e^{-M/n} \le e^{-4/9}$ and $(1-1/n)^M \le 1$ (in this order)
we obtain for all $n\ge 100$ that this is at least
\[
    M\left[1-e^{-4/9} \right]-\frac{M}{n-1}
    = M\left[1-e^{-4/9}-\frac{1}{99}\right]
    \ge \frac{M}3,
\]
where the last step can be verified numerically.

Next we turn to the main case $H_t\geq 1$.
Using $1+x \le e^x$ for $x\in \mathbb{R}$ and $e^{-x}\leq 1-x+x^2/2$ for $x\geq 0$ (in this order) we infer 
	\[
	    1-\left(1-\frac{1}{n}\right)^{U_t}
	    \ge 1 - e^{-U_t/n}\geq \frac{U_t}{n}\left(1-\frac{U_t}{2n}\right).
	\]
Moreover, using Bernoulli's inequality together with the obvious bound $(n-1)^{-1}>n^{-1}$ implies, since $H_t \ge 1$,
\[
    U_t\left(1-\frac{1}{n}\right)^{U_t}\frac{H_t-1}{n-1}
    \ge U_t\left(1-\frac{U_t}{n}\right)\frac{H_t-1}{n}.
\]
By substituting the last two bounds into~\eqref{pghstar} we obtain 
    \begin{equation}\label{septermsstar}
        \mathbb{E}\big[U_{t+1} \mid U_t\big]
        \geq M\frac{U_t}{n}\left(1-\frac{U_t}{2n}\right)
            + U_t\left(1-\frac{U_t}{n}\right)\frac{H_t-1}{n}.
    \end{equation}
Since $M=(1+\varepsilon)n/2$ and $|\varepsilon| < 1/9$, we readily see that
\[
    M\frac{U_t}{n}\left(1-\frac{U_t}{2n}\right)\geq \frac{1+\varepsilon}{2}U_t-\frac{U^2_t}{2n}.
\]
Moreover, since $H_t-1=M-U_t-1$, $M = (1+\varepsilon)n/2$ and $|\varepsilon| < 1/9$, we immediately obtain 
\[
    U_t\left(1-\frac{U_t}{n}\right)\frac{H_t-1}{n}
    \ge \frac{M}{n}U_t\left(1-\frac{U_t}{n}\right)-\frac{U^2_t}{n}-\frac{U_t}{n}\ge
    \frac{1+\varepsilon}{2}U_t-\frac{3U^2_t}{n}.
\]
Substituting these last two estimates back into \eqref{septermsstar} yields the first statement of the lemma. For the second statement, going back to (\ref{septermsstar}) we see that, since $U_t\leq M=(1+\varepsilon)/2$,
    \[\mathbb{E}\big[U_{t+1} \mid U_t\big]
        \geq M\frac{U_t}{n}\left(1-\frac{U_t}{2n}\right)\geq  \frac{1+\varepsilon}{2}U_t\left(1-\frac{1+\varepsilon}{4}\right)\geq \frac{1+\varepsilon}{3}U_t,\]
    where the last inequality follows from our assumption $|\varepsilon| < 1/9$. This concludes the proof.
\end{proof}
Next we state an upper bound for the (conditional) expected value of $U_{t+1}$ given $U_t$.
\begin{lem}\label{expunhappy}
    Let $|\varepsilon|\le 1$ and $M=(1+\varepsilon)n/2$. Then
\begin{equation*}
	\mathbb{E}\big[U_{t+1} \mid U_t\big]
    \leq (1 + \varepsilon)U_t - \frac{U^2_t}n.
\end{equation*}
\end{lem}
\begin{proof}
If $U_t = 0$ then the statement is obvious, so we assume throughout that $U_t \geq 1$. Using Lemma~\ref{mainexpvalueunhappy} we obtain
\begin{align}\label{firstbound}
\nonumber\mathbb{E}\left[U_{t+1} \mid U_t\right]
&\leq U_t-U_t\frac{n-H_t}{n}\left(1-\frac{1}{n}\right)^{U_t}+H_t-H_t\left(1-\frac{1}{n}\right)^{U_t}.
\end{align}
Bernoulli's inequality implies $(1-1/n)^{U_t}\geq 1 - U_t/n$, and hence
\begin{equation*}
\label{firstbound}
\nonumber\mathbb{E}\left[U_{t+1} \mid U_t\right]
\leq U_t-U_t\left(1-\frac{H_t}n\right)\left(1 - \frac{U_t}n\right) + \frac{H_tU_t}n
 = U_t\left(\frac{2H_t + U_t}n - \frac{H_tU_t}{n^2}\right).
\end{equation*}
Since $2H_t + U_t = 2M - U_t = (1+\varepsilon)n - U_t$ we obtain the claimed bound.
\end{proof}

\subsection{Reducing the Number of Unhappy Particles}
\label{driftsubs}

Here we show a lemma that provides information on the number of steps required to  reduce the number of unhappy particles from some number $\ell \leq M$ to some $h< \ell$.
\begin{lem}\label{drifanalysislemma}
    Let $n\in\mathbb{N}$, $t_0\in \mathbb{N}_0$, $|\varepsilon |<1$, $M = (1+\varepsilon)n/2 \in \mathbb{N}$ and $h \geq \max\{2\varepsilon n,1\}>0$. Set $\tau_h\coloneqq \inf \{t\geq t_0: U_t< h \}$. Then, for any $b>0$ and for every $\ell\ge h$ 
	\begin{equation}\label{splitcases}
	    \mathbb{P}\big(\tau_h - t_0 > b \mid U_{t_0}=\ell\big)
        \leq \frac{4n}{bh}.
	\end{equation}
\end{lem} 
\begin{proof}
By Lemma~\ref{expunhappy}, for every $t\ge t_0$ we have  
\[\mathbb{E}\left[U_t-U_{t+1} \mid  U_t=x\right]\geq x^2/n-\varepsilon x.\]
It follows from Theorem \ref{driftthm}, applied to the monotone increasing function
    $f:[h,\infty)\to \mathbb{R}^+, x\mapsto  x^2/2-\varepsilon x$ and $x_{\min} = h$, that
	\begin{equation*}
	\mathbb{E}[\tau_h-t_0]\leq \frac{h}{f(h)}+ \int_{h}^{M}\frac{1}{f(x)}dx.
	\end{equation*}
    For every $x\ge h\ge 2\varepsilon n $ we have $\varepsilon n/x\leq 1/2$ and so
	\begin{equation*}
	  f(x)=\frac{x^2}{n}-\varepsilon x=\frac{x^2}{n}\left(1-\frac{\varepsilon n}{x}\right)\geq \frac{x^2}{2n}, \quad x \ge h.
  \end{equation*}
    Therefore we arrive at
    \[
\mathbb{E}[\tau_h-t_0]\leq \frac{h}{f(h)}+ \int_{h}^{M}\frac{1}{f(x)}dx\le \frac{2n}{h}+\int_{h}^{M}\frac{2n}{x^2} dx\le \frac{4n}{h}.
    \]
The result follows by applying Markov's inequality.
\end{proof}

\subsection{Branching Process Approximation}\label{couplingfacts}
In this subsection we introduce a process that we will use in order to obtain upper and lower couplings for the process of unhappy particles. For $|\kappa|<1$ and some random variable $N \in \mathbb{N}$ define a stochastic process $(Z_t(N,\kappa))_{t\in \mathbb{N}_0}$ by
\begin{equation}
\label{eq:defBinProcess}
    Z_0(N,\kappa) = N
    \quad
    \text{and}
    \quad
    Z_t(N, \kappa)=_d 2 \, \text{Bin}\big(Z_{t-1}(N,\kappa),(1+\kappa)/2\big), ~ t\in\mathbb{N}.
\end{equation}
We call $(Z_t(N,\kappa))_{t\in \mathbb{N}_0}$ the \textit{binomial process} in this paper. We want to stress that the process is well defined even if $\kappa < 0$ (as long as $\kappa > -1$). When considering $Z_t(N,\kappa)$ we sometimes drop the dependence on $N$ or on $\kappa$, if it is clear from the context or otherwise convenient.
In the next lemma we show how to control the number of unhappy particles at some given step in terms of the binomial process.
\begin{lem}\label{lem:uppercoupling}
    Let $|\varepsilon| < 1$. 
    Consider the dispersion process with $M=(1+\varepsilon)n/2$ particles.
    Then
    $$U_{t_0+t}\le_{sd} Z_t(U_{t_0},\varepsilon), \quad t,t_0\in\mathbb{N}_0.$$
    \end{lem}

\begin{proof}
By the definition of the binomial process,
\[
    Z_t =_d 2 \,\sum_{i=1}^{Z_{t-1}}B_{i,t}
    ~\text{  for all  }~ t\in \mathbb{N},
\]
where $(B_{i,t})_{i,t\in \mathbb{N}}$ is a doubly-infinite sequence of iid $\text{Ber}((1+\varepsilon)/2)$ random variables. We claim that, for all $t\geq 1$, 
\begin{equation}\label{stochdom}
U_{t_0+t}\leq_{sd} 2\sum_{i=1}^{U_{t_0+t-1}}B_{i,t}.
\end{equation}
To see this, at step $t_0+t$, remove the particles in ${\cal U}_{t_0+t-1}$ from the graph and put them back one after the other (by choosing for each one of them independently and uniformly at random a vertex) to obtain ${\cal U}_{t_0+t}$.
Note that, during this process, the probability that the $j$-th particle, where $1 \le j \le U_{t_0+t-1}$, is put on one of the, at most $M$, currently occupied vertices is at most~$M/n$. 
Moreover, if this event occurs, then the number of unhappy particles increases by at most~2; it increases by exactly~2 if the $j$-th particle is put on a vertex that is currently occupied by exactly one other particle.
This establishes~\eqref{stochdom}.
With this at hand we proceed by showing the statement of the lemma by induction over $t$. 
Note that, for $t=0$, we have that $U_{t_0} = Z_0$. 
Next, suppose that the statement is true for $t-1$, where $t\in \mathbb{N}$. Then, for any $z \in \mathbb{N}_0$,
\begin{align*}
    \mathbb{P}(U_{t_0+t}\geq z)
    \le \mathbb{P}\left(\sum_{i=1}^{U_{t_0+t-1}}B_{i,t}\geq z/2\right)
    &\le \mathbb{P}\left(\sum_{i=1}^{Z_{t-1}}B_{i,t}\geq z/2\right)=\mathbb{P}(Z_t\geq z).
\qedhere
\end{align*}
\end{proof}
We can also establish a similar lower coupling by using a different parameter for the binomial process, as long as the number of unhappy particles remains sufficiently small. In particular, in the following statement we consider only the case $U_{t_0} \le \delta n$, where $\delta < (1+\varepsilon)/4$, so that the probability in the dominated binomial process~\eqref{eq:defBinProcess} is indeed a probability, that is, it is in $[0,1]$. 
\begin{lem}\label{lem:lowercoupling}
Let $|\varepsilon| < 1$, $0<\delta<(1+\varepsilon)/4$ and consider the dispersion process with $M=(1+\varepsilon)n/2$ particles. Let $t_0\in \mathbb{N}_0$ be such that $U_{t_0}\le \delta n$ and let $\tau=\inf\{t\ge t_0: U_t\ge \delta n\}$.
Then,
$$Z_t(U_{t_0}, \varepsilon - 4\delta) \le_{sd} U_{t_0+t}, \qquad 0\le t \le \tau-t_0.$$
\end{lem}
\begin{proof}
Let $1\leq t\leq \tau-t_0$ and, similarly to the proof of Lemma~\ref{lem:uppercoupling}, remove the unhappy particles in $\mathcal{U}_{t_0+t-1}$ from the graph and put them back one by one (by choosing for each one of them independently and uniformly at random a vertex).
After every unhappy particle has been put back we end up with $\mathcal{U}_{t_0+t}$.
Denote by $\mathcal{Q}_{t_0+t}(p)$ the event that particle $p\in \mathcal{U}_{t_0+t-1}$ is placed on a vertex containing only a single particle.
Then
   $$U_{t_0+t}\ge 2\sum_{p\in \mathcal{U}_{t_0+t-1}}\mathbb{1}_{\mathcal{Q}_{t_0+t}(p)},$$
since if $\mathcal{Q}_{t_0+t}(p)$ occurs, then $p$ is put on a vertex containing a single particle and both of them are unhappy afterwards.
Note that while adding the unhappy particles back onto the graph, the number of vertices containing exactly one particle is always at least $H_{t_0+t-1}-U_{t_0+t-1}$ (since every particle that is put back decreases the number of happy particles by at most one) and 
$$H_{t_0+t-1}-U_{t_0+t-1}= M-2U_{t_0+t-1}=\frac{1+\varepsilon}{2}n-2U_{t_0+t-1},$$
implying
$$
    \mathbb{P}\big(\mathcal{Q}_{t_0+t}(p) \mid U_{t_0+t-1}, \{\mathcal{Q}_{t_0+t}(p'):p'\in U_{t_0+t-1}\setminus \{p\}\}\big)
    \ge \frac{1+\varepsilon}{2}-2\frac{U_{t_0+t-1}}{n}.
$$
If $U_{t_0+t-1}\le \delta n$ by definition of $\tau$, the last expression is at least $ (1+\varepsilon)/2-2\delta \eqqcolon q$. 
Therefore, by induction on $t$, if $(B_{i,t})_{i,t\geq 1}$ denotes a doubly-infinite sequence of iid $\text{Ber}(q)$ random variables,
\begin{equation*}\label{stochdombelow}
U_{t_0+t}\geq_{sd}2\sum_{i=1}^{U_{t_0+t-1}}B_{i,t}\ge_{sd} 2\sum_{i=1}^{Z_{t-1}}B_{i,t}=Z_t.
\qedhere
\end{equation*}
\end{proof}
In order to study the binomial process $(Z_t(N,\kappa))_{t \in \mathbb{N}_0}$ from~\eqref{eq:defBinProcess} we provide an alternative description of it.
We can think of any particle $p$ counted in $Z_0$ to be the root of an independent Galton-Watson tree $\mathcal{T}_p$, where every individual has $2\text{Ber}((1+\kappa)/2)$ offspring, that is, no offspring with probability $(1-\kappa)/2$ and 2 offspring with probability $(1+\kappa)/2$.
Obviously, $Z_t$ equals in distribution the total number of individuals in all $\mathcal{T}_p$'s at distance $t$ from the root.
Let $\mathcal{T}(\kappa)$ be a Galton-Watson tree with the aforementioned distribution, and let $\{\mathcal{T}(\kappa)\le t\}$ be the event that this branching process has at most $t$ levels (in particular, there are no individuals at distance $t$ from the root).
Then we immediately obtain the following observation, tailored to our intended application.
\begin{fact}\label{obs:binbranch}
Let $|\varepsilon|<1$. Then, for any random variable $N$ with support in $\mathbb{N}$, we have
$$\mathbb{P}(Z_t(N, \varepsilon) = 0 \mid N) = \mathbb{P}(\mathcal{T}(\varepsilon)\le t)^{N}, \quad t \in \mathbb{N}_0.$$
\end{fact}
The following simple lemma provides bounds for the probability that a Galton-Watson tree (having an offspring distribution supported on two integers) has at least $k$ levels. 
\begin{lem}\label{BPgenrallemma1}
Let $|\varepsilon| <1$. Then for any $k\in \mathbb{N}$ 
	\begin{equation}\label{bound1}
	    \mathbb{P}(\mathcal{T}(\varepsilon)\geq k)\leq \frac{2\varepsilon}{1-(1-2\varepsilon)(1+\varepsilon)^{-k+1}}
	\end{equation}
	and
	\begin{equation}\label{bound2}
	    \mathbb{P}(\mathcal{T}(\varepsilon)\geq k)\geq \frac{\varepsilon}{1-(1-\varepsilon)(1+\varepsilon)^{-k+1}}.
	\end{equation} 
\end{lem}
\begin{proof}
Abbreviate $x_k\coloneqq \mathbb{P}(\mathcal{T}(\varepsilon)\geq k)$ for $k\in\mathbb{N}$. Since $\cal T(\varepsilon)$ has at least $k$ levels only if the root has degree two and at least one of the subtrees has at least $k-1$ levels, we obtain $x_1 = 1$ and
\[
    x_k
    = \frac{1}{2}\left(1+\varepsilon\right)(2x_{k-1}-x^2_{k-1})
	= (1+\varepsilon)x_{k-1}\left( 1- \frac{x_{k-1}}{2}\right), \quad k \ge 2.
\]
Setting $y_k\coloneqq 1/x_k$ we obtain for $k \ge 2$
\begin{equation}\label{eq:yexpansion}
    y_k
    = \frac{y_{k-1}}{1+\varepsilon}\left( 1- \frac{x_{k-1}}{2}\right)^{-1}
    \ge
    \frac{y_{k-1}}{1+\varepsilon}\left(1+\frac{x_{k-1}}2\right)
    =
    \frac1{1+\varepsilon}(y_{k-1} + 1/2).
\end{equation}
Then, by solving the linear recursion we obtain
\begin{equation*}
\label{indforyk}
y_k\geq (1+\varepsilon)^{-k+1}+\frac12\sum_{1 \le j \le k-1}(1+\varepsilon)^{-j},
\quad k \in\mathbb{N}.
\end{equation*}
Noting that 
\begin{equation}\label{geomsum}
    \sum_{1 \le j \le k-1}(1+\varepsilon)^{-j}
    = \frac{1-(1+\varepsilon)^{-k+1}}{\varepsilon},
\end{equation}
we then obtain~\eqref{bound1}.
The lower bound is established in a similar way. Indeed, we obtain from $y_k = \frac{y_{k-1}}{1+\varepsilon}(1-x_{k-1}/2)^{-1}$ (see the first `$=$' in \eqref{eq:yexpansion}) using $\sum_{j=2}^{\infty}q^j\leq 2q^2$ for all $q\in [0,1/2]$ that
\begin{equation*}
    y_k
    \leq y_{k-1}(1+\varepsilon)^{-1}(1+x_{k-1}/2+x^2_{k-1}/2)
    \leq (1+\varepsilon)^{-1}(1+y_{k-1}). 
\end{equation*}
Solving again a linear recursion then yields
\[
    y_k \leq (1+\varepsilon)^{-k+1} + \sum_{1 \le j \le k-1}(1+\varepsilon)^{-j},
    \quad k\in\mathbb{N},
\]
and~\eqref{bound2} follows as well, again using~\eqref{geomsum}.
\end{proof}

\section{The Upper Tail}
\label{sec:uppertail}
Throughout this section we assume that $M=(1+\varepsilon)n/2$, where $|\varepsilon| = o(1)$, $|\varepsilon| < 1/9 $ and that, say, $n \ge 100$, so that the statements from the previous section apply. Define
$$
    \hat\varepsilon\coloneqq \max\{ |\varepsilon|, e n^{-1/2}\}
$$
that we will mostly use without further reference.
The main goal of this section is to prove the following statement, from which the bounds stated in Theorem \ref{mainthm} follow immediately. No attempt was made to optimize the constants.
\begin{prop}\label{mainpropsec3}
Let $A\ge 1$. If $\varepsilon< -e n^{-1/2}$, then
$$\mathbb{P}\big(T_{n,M}> A e^{45} \hat{\varepsilon}^{-1}\ln(\hat{\varepsilon}^2 n)\big)\le e^{-(A-1)}.$$
Otherwise, if $\varepsilon \ge -e n^{-1/2}$, then
$$\mathbb{P}\big(T_{n,M} > A \hat{\varepsilon}^{-1} e^{2^{10} \hat{\varepsilon}^2n} \big)\le e^{-(A-1)}.$$
\end{prop}
\noindent
The first step in our proof  establishes, by means of the following lemma, that if there is a positive probability, say $q$, for the dispersion process to finish in less than $k$ steps, given that it started from any number $s\in [M]$ of unhappy particles, then it is (exponentially) unlikely for the dispersion process to run for more than $Aq^{-1}k$ steps. 
\begin{lem}\label{lem:expbound}
Suppose that there is a $q > 0$ and a $k\in\mathbb{N}$ such that for any $t_0\in \mathbb{N}_0$ 
\begin{equation}\label{unifbound}
    \min_{s\in [M]}\mathbb{P}(T_{n,M}\le t_0+k\mid U_{t_0}=s)\ge q.
\end{equation}
Then for all $A > 0$ 
\[\mathbb{P}(T_{n,M}> Aq^{-1} k)\le e^{-(A-1)}.\]
\end{lem}
\begin{proof}
Set $I_i=((i-1)k, ik]$ for all $i\in \mathbb{N}$ so that $\mathbb{N}$ is the union of all those intervals. Note that if~$T_{n,M} > Aq^{-1}k$, then necessarily the process of unhappy particles cannot terminate in any of the  intervals $I_i$ with $i \le A q^{-1}$. Denote by $\mathcal{S}_i$ the event that the process of unhappy particles \textit{does not} terminate until the end of $I_i$; that is,
\begin{equation*}
    \mathcal{S}_i \coloneqq \{U_t>0 \text{ for all } t\in I_i\}, \text{ }i\in \mathbb{N}.
\end{equation*}
So,
\begin{align}\label{befkeylemma}
    \mathbb{P}(T_{n,M}>Aq^{-1} k)
    \le \prod_{1\le i \le \lfloor Aq^{-1} \rfloor }\mathbb{P}\big(\mathcal{S}_i \mid \cap_{r=1}^{i-1}\mathcal{S}_r\big).
\end{align}
Now thanks to our hypothesis~\eqref{unifbound} and the Markov property of the process of unhappy particles 
\[
    \mathbb{P}\big(\mathcal{S}_i \mid  \cap_{r=1}^{i-1}\mathcal{S}_r\big)
    =\mathbb{P}\big(\mathcal{S}_i \mid  U_{(i-1)k}>0\big)
= 1-\mathbb{P}(T_{n,M}\le ik \mid U_{(i-1)k}>0)\le 1-q \le e^{-q}.
\]
Substituting this into~\eqref{befkeylemma} finishes the proof. 
\end{proof}
The remainder of this section is devoted to establishing~\eqref{unifbound} for some appropriate $k$ and $q$.
Recall that $\hat{\varepsilon} = \max\{|\varepsilon|,en^{-1/2}\}$. The next statement basically asserts that if there is a positive (conditional) probability, say $q$, that the algorithm stops in less than $k/2$ steps, for some $k\geq \hat{\varepsilon}^{-1}$, given that we started with $r< 16\hat{\varepsilon}n$ unhappy particles, then with (conditional) probability at least $q/2$ the process will finish within less than $k$ steps. 
\begin{lem}\label{lem:driftanalysisappl}
If $k\ge \hat{\varepsilon}^{-1}$ and $q\in \mathbb{R}$ satisfy for any $t_0\in \mathbb{N}_0$
\begin{equation}\label{hplem}
    \min_{r < 16\hat{\varepsilon}n} \mathbb{P}(T_{n,M}\le t_0+k/2\mid U_{t_0}=r)\ge q,
\end{equation}
then we have for any $t_0\in \mathbb{N}_0$
$$\min_{s\in [M]}\mathbb{P}(T_{n,M}\le t_0+k\mid U_{t_0}=s)\ge q/2.$$
\end{lem}
\begin{proof}
Note that if there is an $m\leq k/2$ such that $U_{t_0+m}< 16\hat{\varepsilon}n$ and $U_{t_0+m+k/2}=0$, then clearly we must have $T_{n,M}\leq t_0+k$, too. Consequently, 
\begin{equation*}
    \mathbb{P}(T_{n,M}\le t_0+k\mid U_{t_0}=s)
    \geq
    \mathbb{P}\big(\exists 
 m\leq k/2:U_{t_0+m}< 16\hat{\varepsilon}n, U_{t_0+m+k/2}=0\mid U_{t_0}=s\big).
\end{equation*}
We shall also use a second observation. Let $0\le r < 16 \hat{\varepsilon}n$ and $0 \le m \le k/2$. Our assumption~\eqref{hplem} and the Markov property of the process imply
\begin{align*}
\mathbb{P}(U_{t_0+m+k/2} = 0\mid U_{t_0+m}=r)
\ge \mathbb{P}(T_{n,M} \le t_0+k/2\mid U_{t_0}=r)
\ge q.
\end{align*}
Let $\tau \ge t_0$ be the first time such that $U_\tau < 16\hat{\varepsilon}n$. By conditioning on $\tau$ and $U_\tau$ and using the Markov property of the process once more we obtain from the two observations that
\begin{align*}
    \mathbb{P}(T_{n,M}\le &~ t_0 +  k  \mid U_{t_0}=s) \\ 
    &\geq q \sum_{0 \le m \le k/2}  \sum_{0 \le r < 16\hat{\varepsilon}n} \mathbb{P}(\tau=t_0+m\mid U_{t_0}=s) \mathbb{P}(U_\tau = r\mid U_{t_0} = s, \tau=t_0+m)\\
    & = q\sum_{0 \le m \le k/2}\mathbb{P}(U_{\tau} < 16\hat{\varepsilon}n, \tau = t_0 + m \mid U_{t_0} = s)\\
    & = q \cdot \mathbb{P}(\tau\le t_0+ k/2 \mid U_{t_0}=s).
\end{align*} 
Note that if $s< 16 \hat{\varepsilon}n$, then $\tau=t_0$ with probability 1. Otherwise,
by applying Lemma~\ref{drifanalysislemma}, 
\begin{align*}
    \mathbb{P}(\tau\le t_0 + k/2 \mid U_{t_0}=s)
    \geq
    1-\frac{4n/(16\hat{\varepsilon}n)}{k/2}\ge \frac{1}{2},
\end{align*}
as $k\ge \hat{\varepsilon}^{-1}$, completing the proof.
\end{proof}
The last remaining piece is to establish a lower bound for the probability that the process starting from $r< 16\hat{\varepsilon}n$ unhappy particles terminates in the following $k/2$ steps; that is, we need to verify~\eqref{hplem} (and so, also~\eqref{unifbound}). To this end we will consider the cases $\varepsilon<-e n^{-1/2}$ and $\varepsilon\ge -e n^{-1/2}$ separately.
\begin{lem}\label{lem:quickdiesmalleps}
Let $t_0 \in \mathbb{N}_0$. Suppose that $\varepsilon<-en^{-1/2}$. Then 
$$\min_{r< 16 \hat{\varepsilon}n}
    \mathbb{P}\big(T_{n,M}\le t_0+\hat{\varepsilon}^{-1}\ln(\hat{\varepsilon}^2 n) \mid U_{t_0}=r\big)
\ge e^{-43},
\quad
t_0 \in\mathbb{N}.$$
\end{lem}
\begin{proof}
Recall that $U_t=0$ is equivalent to $T_{n,M}\le t$. Also, recall the definition of the binomial process that we introduced in~\eqref{eq:defBinProcess}. Throughout this proof we let $k\coloneqq \lfloor\hat{\varepsilon}^{-1}\ln(\hat{\varepsilon}^2 n)\rfloor$. By applying Lemma~\ref{lem:uppercoupling} for any $1 \le r < 16 \hat{\varepsilon}n$, we obtain 
\begin{equation*}
\mathbb{P}(T_{n,M}\le t_0+k\mid U_{t_0} =r)
=\mathbb{P}(U_{t_0+k}=0 \mid U_{t_0} = r)
\ge \mathbb{P}(Z_{k}(r, \varepsilon) = 0).
\end{equation*}
Thanks to Fact~\ref{obs:binbranch}, in order to provide a lower bound on the probability which appears in the last display, we only need to consider the probability that the Galton-Watson process $\mathcal{T}(\varepsilon)$ has more than $k$ levels.  Namely, we need to bound from below
\begin{equation}
    \mathbb{P}(\mathcal{T}(\varepsilon) \le k)
    =
    1-\mathbb{P}(\mathcal{T}(\varepsilon)\ge  k+1).
\end{equation}
Note that, since $\varepsilon < -en^{-1/2}$, then $|\varepsilon|>en^{-1/2}$ and thus $-\varepsilon=\hat{\varepsilon}$. 
By applying Lemma~\ref{BPgenrallemma1} we obtain
\begin{equation}\label{kl}
    \mathbb{P}(\mathcal{T}(\varepsilon)\ge k+1)
    \le \frac{2\varepsilon}{1-(1-2\varepsilon) (1+\varepsilon)^{-k}}
    = -\frac{2\hat{\varepsilon}}{1-(1+2\hat{\varepsilon})(1-\hat{\varepsilon})^{-k}}.
\end{equation}
Since $1-x \le e^{-x}$, we have $(1-\hat{\varepsilon})^{-k}\geq (1-\hat\varepsilon)e^{(k+1)\hat{\varepsilon}}$. As $(1-\hat\varepsilon)(1+2\hat\varepsilon) \ge 1$ for $\hat\varepsilon\in[0,1/2]$
\begin{equation*}
    1-(1+2\hat{\varepsilon})(1-\hat{\varepsilon})^{-k}
    \le 1-e^{(k+1)\hat{\varepsilon}}
    \le -(\hat{\varepsilon}^2n-1),
\end{equation*}
where the last inequality follows from $k\ge \hat{\varepsilon}^{-1}\ln(\hat{\varepsilon}^2 n)-1$. Therefore, since
$(\hat{\varepsilon}^2n)^{-1}\leq e^{-2} < 1/5$,
we see that the ratio on the right-hand side of (\ref{kl}) is at most
\[
    \frac{2\hat{\varepsilon}}{\hat{\varepsilon}^2n-1}\leq \frac{2\hat{\varepsilon}}{\hat{\varepsilon}^2n(1-e^{-2})}
    \le \frac52 \cdot \frac1{\hat{\varepsilon}n}.
\]
Finally, since $(1-xN^{-1})^N \ge (1-xN^{-1})e^{-x}$ for $x \ge 0$, we conclude that with probability at least 
$$
    \left(1- \frac52 \cdot \frac1{\hat{\varepsilon}n}\right)^r
    \ge \left(1- \frac52 \cdot \frac1{\hat{\varepsilon}n}\right)^{16 \hat{\varepsilon}n}
    \ge \left(1- \frac52 \cdot \frac1{\hat{\varepsilon}n}\right)e^{-40}.
$$
all $Z_0 = r$ trees die out by level $k$. Since $\hat\varepsilon n \ge e$ we further obtain $1- 5/(2\hat{\varepsilon}n) \ge 1 - 5/(2e)$ and this value is, say, at least $ e^{-3}$. The claim follows.
\end{proof}
Next we consider the case $\varepsilon \ge - e n^{-1/2}$.
\begin{lem}\label{lem:quickdielargeeps}
Let $t_0 \in \mathbb{N}_0$.
Suppose that $\varepsilon \ge -e n^{-1/2}$. Then
$$\min_{r < 16 \hat{\varepsilon}n}\mathbb{P}(T_{n,M}\le t_0+\hat{\varepsilon}^{-1}/2\mid U_{t_0}=r)\ge e^{-2^9\hat{\varepsilon}^2 n}, \quad t\in\mathbb{N}_0.
$$
\end{lem}
\begin{proof}
Similarly to the proof of the previous lemma we estimate the probability that all $r < 16\hat{\varepsilon}n$ independent copies of a Galton-Watson branching process $\mathcal{T}(\varepsilon)$ die out by generation $k\coloneqq  \lfloor\hat{\varepsilon}^{-1} /2\rfloor$. Recall from Lemma~\ref{BPgenrallemma1} that
\begin{equation}\label{claimb}
    \mathbb{P}(\mathcal{T}(\varepsilon)\ge k+1)\leq \frac{2\varepsilon}{1-(1-2\varepsilon)(1+\varepsilon)^{-k}}.
\end{equation}
Recall that $|\varepsilon| < 1/2$.
Using Bernoulli's inequality,
\[
    1-(1-2\varepsilon)(1+\varepsilon)^{-k}
    \ge 1-\frac{1-2\varepsilon}{1+k\varepsilon}
    = \frac{(2+k)\varepsilon}{1+k\varepsilon}.
\]
Note that if $\varepsilon < 0$ then $1+k\varepsilon \le 1$ and otherwise $1+k\varepsilon \le 2$. So we obtain from~\eqref{claimb} that $\mathbb{P}(\mathcal{T}(\varepsilon)\ge k+1) \le 8\hat\varepsilon$.
Therefore the probability that all $r < 16\hat{\varepsilon}n$ trees die out is at least (with room to spare, since, say, $1-x \ge e^{-4x}$ for $x\in[0,8/9]$)
$(1-8\hat{\varepsilon})^{16\hat{\varepsilon}n}\ge e^{-2^{9} \hat{\varepsilon}^2 n}$, and the claim follows.
\end{proof}

\paragraph{Proof of Proposition \ref{mainpropsec3}.}
The first statement follows from Lemma~\ref{lem:expbound},
by setting
$
    k = 2 \lfloor \hat{\varepsilon}^{-1}\ln(\hat{\varepsilon}^2  n) \rfloor
$
and choosing $q = 2e^{-45}$ by applying Lemmas~\ref{lem:driftanalysisappl} and \ref{lem:quickdiesmalleps}. Similarly, the second statement also follows from from Lemma~\ref{lem:expbound}, by setting this time $k= \lfloor \hat{\varepsilon}^{-1} \rfloor$ and choosing $q= 2e^{-2^{10}\hat{\varepsilon}^2n}$ by applying Lemmas~\ref{lem:driftanalysisappl} and~\ref{lem:quickdielargeeps}. 
\qed

\section{The Lower Tail}
\label{sec:lowertail}
Throughout this section we assume that $M=(1+\varepsilon)n/2$, where $|\varepsilon| = o(1)$, $|\varepsilon| < 1/9 $ and that, say, $n \ge 100$, so that the statements from the Section~\ref{prel} apply. Define, like previously, 
\[
    \hat{\varepsilon} \coloneqq \max\{|\varepsilon|, e n^{-1/2}\}.
\]
We show the following statement, from which Theorem~\ref{mainthm2} follows immediately.
\begin{prop}\label{mainpropsec4}
For sufficiently large $n$ the following is true. If $\varepsilon\le e n^{-1/2}$ then 
$$
    \mathbb{P}\big(T_{n,M}\le \hat{\varepsilon}^{-1}\ln(\hat{\varepsilon}^2 n)/ A\big)
    \le 
    3\exp(-A/2^{19}), \quad A \ge 1.
$$
On the other hand, if $\varepsilon > e n^{-1/2}$, then with $ k_0\coloneqq e^{2^{-16}\hat{\varepsilon}^2 n}$ we obtain 
$$\mathbb{P}\big(T_{n,M} \le \hat{\varepsilon}^{-1}k_0/(2A)\big)\le
\left\{
\begin{array}{ll}
      3\exp\left(-\frac{A \hat{\varepsilon}^2 n}{2^{18} k_0 }\right) & \mbox{, if } A>k_0 \\
      5A^{-1} &\mbox{, if } 1 \le A\le k_0
\end{array}
\right. \enspace .
$$
\end{prop}
An important tool in the proof of Proposition \ref{mainpropsec4} is the coupling of the number of unhappy particles with the binomial process introduced in Section \ref{prel}, see Lemma~\ref{lem:lowercoupling}. As we will see shortly, applying this coupling requires a more involved argument than in the previous section, since now we need (among other things) to carefully balance the following two conditions.
On the one hand, the coupling only works once the number of unhappy particles becomes sufficiently \emph{small}. However, the required bounds on the lower tail of $T_{n,M}$ can only be achieved if, at the start of the coupling, the number of unhappy particles is still sufficiently \emph{large}. In particular, if $U_t$ would make `big jumps', then the process could completely bypass the regime where the previously described balance is achieved. To avoid this, for $b\in\mathbb{N}$ we define the (good) event
\begin{equation}\label{eventa}
    \mathcal{A}(b)
    \coloneqq
    \left\{U_{t+1}\geq 2^{-5} U_{t} \text{ for all } t\leq b \text{ such that } U_{t}>2^{-5}\hat{\varepsilon} n 
    \right\}.
\end{equation}
The next lemma asserts that $\mathcal{A}(b)$ occurs for $b = \exp(O(\hat\varepsilon n))$ with (very) high probability, meaning that it is very unlikely for $U_{t+1}$ to drop below $2^{-5} U_t$ if $U_t>2^{-5}\hat{\varepsilon} n$, for all $t\le b$.
\begin{lem}\label{boundbadevent2}
If $b \le e^{2^{-19}\hat{\varepsilon}n}$, then 
\begin{equation*}
		\mathbb{P}\big(\mathcal{A}(b)^c\big)
		\le \exp \big(- 2^{-19}\hat{\varepsilon} n\big).
		\end{equation*}
	\end{lem}
\begin{proof}
Using the second statement of Lemma~\ref{expunhappyupper} we obtain for $n \ge 100$ 
\begin{equation*}
\mathbb{E}[U_{ t+1} \mid U_{ t}]\geq \frac{1+\varepsilon}{3}U_t \geq 2^{-2}U_{t}.
\end{equation*}  
Then a union bound yields (with plenty of room to spare)
\begin{align}\label{plugazuma}
	\nonumber\mathbb{P}(\mathcal{A}(b)^c)&\leq \sum_{1 \le t \le b}\mathbb{P}\left(U_{ t}>2^{-5}\hat{\varepsilon} n, U_{ t+1}<2^{-5}U_{ t}\right)\\
	\nonumber&\leq \sum_{1 \le t \le b}\mathbb{P}\left(U_{ t}>2^{-5}\hat{\varepsilon} n, U_{ t+1}-\mathbb{E}[U_{ t+1} \mid U_{ t}]<-2^{-5} U_{ t}\right)\\
	&= \sum_{1 \le t \le b} \mathbb{E}\left[\mathbb{1}_{\{U_{ t}>2^{-5}\hat{\varepsilon} n\}}\mathbb{P}\left(U_{t+1}-\mathbb{E}[U_{ t+1}|U_{t}]<-2^{-5}U_{ t} \mid U_{t}\right)\right].
	\end{align} 
    Without loss of generality assume that the set of unhappy particles is $\mathcal{U}_t=\{p_1,\ldots,p_{U_t}\}$ (recall that ${\cal P} = \{p_1, \dots, p_M\}$ is the set of particles). 
    Then we define the Doob martingale 
    \[
        M_i\coloneqq \mathbb{E}_{p_{1,t+1},\dots,p_{i,t+1}}\big[U_{t+1} \mid U_{t}\big], \quad 1\le i \le U_t,
    \]
    and set $M_0=\mathbb{E}[U_{t+1} \mid  U_t]$. In words, $M_i$ is the (conditional) expectation of $U_{ t+1}$ given $U_{ t}$, knowing the values of $p_{1,t+1},\dots,p_{i,t+1}$, i.e.\ the movements of the first $i$ particles that were unhappy at step~$t$. Note that $M_{U_t} = U_{t+1}$ and, moreover, $|M_{i+1}-M_{i}|\leq 2 \eqqcolon c_i$ for $1\leq i\leq U_t$. Therefore we can apply Azuma-Hoeffding's inequality to obtain
	\begin{equation*}
	\mathbb{P}\left(U_{ t+1}-\mathbb{E}[U_{ t+1} \mid U_{ t}]<-2^{-5} U_{ t} \mid U_{ t}\right)
	\leq \exp\left(-\frac{2^{-10} U_{t}}{8} \right),
\end{equation*}
and this is at most $\exp(-2^{-18}\hat{\varepsilon} n)$ if $U_t > 2^{-5}\hat{\varepsilon} n$.
Substituting this last estimate into (\ref{plugazuma}) yields
\begin{equation*}
    \mathbb{P}(\mathcal{A}(b)^c)
    \le \exp\left({2^{-19}\hat{\varepsilon} n} -2^{-18}\hat{\varepsilon} n \right)
    =\exp\left(-2^{-19}\hat{\varepsilon} n \right),
\end{equation*}
and the result follows.
\end{proof}
\noindent
The second issue arising from using the coupling in Lemma~\ref{lem:lowercoupling} between $U_t$ and the binomial process is that it only works over multiple time steps when the number of unhappy particles remains small throughout.
In particular, the coupling provides good estimates only when the number of unhappy particles remains below $\delta n$, where, say, $\delta = \hat{\varepsilon}/8$.
In order to ensure that there are sufficiently many time steps until the number of unhappy particles goes above  $\delta n$, we introduce a second barrier at $\hat{\varepsilon}n/32$, where we initiate the coupling. More specifically, we define $\tau^U_0\coloneqq 0$ and iteratively, for $i\in \mathbb{N}$,
\begin{equation}
    \tau^L_i\coloneqq \inf\big\{t\geq \tau^U_{i-1}+1:U_t\leq \hat{\varepsilon} n/32\big\},
    \text{ }
    \tau^U_i\coloneqq \inf\big\{t\geq \tau^L_i+1:U_t>\hat{\varepsilon} n/8\big\}.
\end{equation}
In words, the stopping times $\tau^L_i$ keep track of the (random) times at which the number of unhappy particles drop below $\hat{\varepsilon} n/32$, whereas the stopping times $\tau^U_i$ record the times at which the process of unhappy particles exceeds $\hat{\varepsilon} n/8$. 
Clearly, the stopping times that we have just defined are finite until the time at which the process terminates.
In particular, for $i\in \mathbb{N}$, if $\tau^U_i<\infty$, then $T_{n,M} > \tau^U_i$ as well.
The next statement establishes that $\tau_i^U - \tau_i^L$ for $i\in \mathbb{N}$ is rather large. In what follows we denote ${\cal F}_t \coloneqq \sigma(\{U_s : 0 \le s \le t\})$ for $t\in \mathbb{N}_0$.
\begin{lem}
\label{lem:upperboundunhappysmalleps}
Let $i,j\in \mathbb{N}$ be such that $j \le \varepsilon^{-1}/2$ when $\varepsilon>0$ (and arbitrary otherwise). On the event $\{\tau_i^L < \infty\} \in \mathcal{F}_{\tau_i^L}$, we have
$$
    \mathbb{P}\big(\tau_i^U-\tau_i^L \ge j \mid \mathcal{F}_{\tau_i^L} \big )
    \ge
    1- \exp\left(-\frac{\hat{\varepsilon} n}{2^{10} j}\right).
$$
\end{lem}
\begin{proof}
Once again, it will be convenient to expose the new positions of the unhappy particles at the beginning of a time step one after the other, and then we repeat the same procedure at the following time steps.
More precisely, we specify a process that stochastically dominates the number of unhappy particles and then show that this process remains below $\hat{\varepsilon}n/8$ for the required number of steps. 
To this end, on $\{\tau_i^L < \infty\}$, we set
\[
    R_0\coloneqq U_{\tau_i^L} ~\text{ and }~ R_\ell\coloneqq R_{\ell-1}+D_\ell  \text{ for }\ell\in \mathbb{N},
\]
where the $D_\ell$'s are iid such that $D_\ell=1$ with probability $(1+\varepsilon)/2$ and $D_\ell=-1$ with probability $(1-\varepsilon)/2$. Define
\[
    S_k := U_{\tau^L_i}+\cdots+ U_{\tau^L_i+k-1},
    \quad k\in \mathbb{N}.
\]
Since the probability that an unhappy particle remains unhappy after moving to a randomly chosen vertex is at most $M/n = (1+\varepsilon)/2$, we note that  $U_{\tau^L_i+1}\leq _{sd}R_{S_1}$  and in general
\[
    U_{\tau^L_i+k} \leq_{sd} R_{S_k}.
\]
Define $\sigma_R$ to be the first \textit{substep} at which the process $R_\ell$ is larger than $\hat{\varepsilon}n/8$; that is,
\[\sigma_R\coloneqq \min\{\ell\in \mathbb{N}:R_\ell\geq \hat{\varepsilon}n/8\}.\]
Observe that, due to the stochastic domination, the probability that there is a step $t < j$ at which $U_{\tau_i^L+t}$ is larger than $\hat{\varepsilon}n/8$ is at most the probability that there exists an $\ell\leq j\hat{\varepsilon}n/8$ at which $R_\ell$ is larger than $\hat{\varepsilon}n/8$.
Therefore the probability in the statement of the lemma is at least
\begin{equation}
\label{eq:tmptiLtuL}    
    1 - \mathbb{P}(\exists \ell\leq j\hat{\varepsilon}n/8 : R_\ell\geq \hat{\varepsilon}n/8).
\end{equation}
In order to prove the lemma we define a martingale in terms of $R_\ell$ and use the Azuma-Hoeffding inequality. Note that $\mathbb{E}[R_{\ell+1} \mid R_\ell]=R_\ell+\varepsilon$ so that $M_\ell\coloneqq R_\ell-\ell\varepsilon$ is a martingale,
and so $M_{\min\{\ell,\sigma_R\}}$ is also a martingale. Observe that
\begin{align}\label{turnintomg}
    \nonumber\mathbb{P}\big(\exists \ell\leq j \hat{\varepsilon}n/8:R_\ell\geq \hat{\varepsilon}n/8\big)
    &\leq \mathbb{P}\big(R_{\min\{(j\hat{\varepsilon}n)/8,\sigma_R\}}\geq \hat{\varepsilon}n/8\big)\\
    \nonumber&=\mathbb{P}\big(M_{\min\{(j\hat{\varepsilon}n)/8,\sigma_R\}}\geq \hat{\varepsilon}n/8-\min\{(j\hat{\varepsilon}n)/8,\sigma_R\}\varepsilon \big).
\end{align}
Using that $\min\{(j\hat{\varepsilon}n)/8, \sigma_R\}\leq j\hat{\varepsilon}n/8$
together with the assumption $j\le \varepsilon^{-1}/2$ when $\varepsilon>0$ and $j\ge 1$ otherwise, we deduce 
\[
    \hat{\varepsilon}n/8-\min\{(j\hat{\varepsilon}n)/8, \sigma_R\}\varepsilon\geq \hat{\varepsilon}n/8- j \hat{\varepsilon}n \varepsilon/8 \ge \hat{\varepsilon}n/16.
\]
We infer
\begin{equation*}
\label{turnintomg}
    \mathbb{P}\big(\exists \ell\leq j \hat{\varepsilon}n/8:R_\ell\ge \hat{\varepsilon}n/8\big)
    \le \mathbb{P}\big(M_{\min\{(j\hat{\varepsilon}n)/8,\sigma_R\}}\ge \hat{\varepsilon}n/16\big).
\end{equation*}
Since $M_0 = R_0 = U_{\tau_i^L}\le \hat{\varepsilon}n/32$ we obtain
\begin{equation*}
    \mathbb{P}\big(M_{\min\{(j\hat{\varepsilon}n)/8, \sigma_R\}}\ge \hat{\varepsilon}n/16 \big)
    \le \mathbb{P}\big(M_{\min\{(j\hat{\varepsilon}n)/8, \sigma_R\}}-M_0\geq \hat{\varepsilon}n/32\big).
\end{equation*}
Using the fact that the martingale differences are bounded from above by $1+|\varepsilon|\leq 2$ (since $|\varepsilon| \le 1/9$), we can use the Azuma-Hoeffding inequality (with $c_i = 2$) to obtain
\begin{align*}
    \mathbb{P}\big(M_{\min\{(j\hat{\varepsilon}n)/8, \sigma_R\}}-M_0\geq \hat{\varepsilon}n/32\big)
    &\leq \exp\left(-{\hat{\varepsilon} n}/{(2^{10}j)}\right),
\end{align*}
completing with~\eqref{eq:tmptiLtuL} the proof of the lemma.
\end{proof}
In the following lemma we first consider the case $\varepsilon\le e n^{-1/2}$. 
Recall that $\tau^L_1$ is the first time $t\in \mathbb{N}$ at which the number of unhappy particles drops below $\hat{\varepsilon}n/32$.
\begin{lem}
\label{lem:boundsmalleps}
Suppose that $\varepsilon\le en^{-1/2}$ and set $b \coloneqq e^{2^{-19}\hat{\varepsilon}n}$. Consider the event 
\[
    \mathcal{S}_1
    \coloneqq
    \big\{2^{-10}\hat{\varepsilon}n \le U_{\tau_1^L}\le 2^{-5} \hat{\varepsilon}n, \tau_1^L \le b \big\}
    \in {\cal F}_{\tau_1^L} ~.
\]
On the event $\mathcal{S}_1$ we have
\[
    \mathbb{P}\left(T_{n,M}
    \le \tau_1^L+\hat{\varepsilon}^{-1}\ln(\hat{\varepsilon}^2 n)/A \mid \mathcal{F}_{\tau_1^L}\right)
    \le 2\exp(-A/2^{13}),
    \quad A \ge 1.
\]
\end{lem}
\begin{proof}
For brevity we will write
\[
    k\coloneqq \lfloor \hat{\varepsilon}^{-1}\ln(\hat{\varepsilon}^2 n)/A \rfloor
\]
throughout. We assume that $k \ge 1$ and $A \ge 4$, as otherwise there is nothing to show.  Note that if $T_{n,M} \le \tau_1^L + k$, then either $k\le \tau_{1}^U-\tau_1^L$ and the process stops within~$k$ steps after $\tau_1^L$, or $k>\tau_{1}^U-\tau_1^L$. Since, by Lemma~\ref{lem:lowercoupling}, the binomial process $ Z_t(U_{\tau_1^L},\varepsilon-\hat{\varepsilon}/2)$  provides a lower coupling for $U_{\tau_1^L+t}$ until time $\tau_1^U$, we obtain on ${\cal S}_1$
\[
    \mathbb{P}\big(T_{n,M}\le \tau_1^L+ k\mid \mathcal{F}_{\tau^L_1}\big) 
    \le \mathbb{P}\big(Z_k(U_{\tau_1^L},\varepsilon-\hat{\varepsilon}/2)=0\mid \mathcal{F}_{\tau^L_1}\big)+\mathbb{P}\big(\tau_1^U-\tau_1^L < k\mid \mathcal{F}_{\tau^L_1}\big).
\]
Hence, on the event ${\cal S}_1$ where $U_{\tau_1^L} \ge 2^{-10}\hat{\varepsilon}n$ and writing
$Z'_t \coloneqq Z_t(\lceil 2^{-10}\hat{\varepsilon}n \rceil,\varepsilon-\hat{\varepsilon}/2)$, we obtain
\begin{align}\label{eq:starblue}
    \mathbb{P}\big(T_{n,M}\le \tau_1^L+ k\mid \mathcal{F}_{\tau^L_1}\big) 
     \le \mathbb{P}\big(Z'_k=0)
    +\mathbb{P}\big(\tau_1^U-\tau_1^L < k\mid \mathcal{F}_{\tau^L_1}\big).
\end{align}
Then it follows from Fact~\ref{obs:binbranch} that  the first term in \eqref{eq:starblue} is at most
$$
    \mathbb{P}\big(Z'_k =0)
    \le \mathbb{P}\big(\mathcal{T}(\varepsilon-\hat{\varepsilon}/2) \le k\big)^{2^{-10}\hat{\varepsilon}n}.
$$
Since $\varepsilon-\hat{\varepsilon}/2\ge -2\hat{\varepsilon}$ (with room to spare), by applying Lemma~\ref{BPgenrallemma1} we obtain 
\begin{equation}\label{LWBblue}
    \mathbb{P}(\mathcal{T}(\varepsilon-\hat{\varepsilon}/2) > k)
    \ge \mathbb{P}(\mathcal{T}(-2\hat{\varepsilon}) > k)
    \ge \frac{2\hat{\varepsilon}}{(1-2\hat{\varepsilon})^{-k}(1+2\hat{\varepsilon})-1}.
\end{equation}
Using $1-x\geq e^{-2x}$ for $x\in [0,1/2]$ and substituting the value of $k$ (where we drop the $\lfloor~\rfloor$) implies 
\[
  (1-2\hat{\varepsilon})^{-k}(1+2\hat{\varepsilon})
  \le (1+2\hat{\varepsilon})e^{4\hat{\varepsilon}k}
   \le (1+2\hat{\varepsilon}) (\hat{\varepsilon}^2n)^{4/A}.
\]
Since $A\ge 4$, by applying Bernoulli's inequality $(1+x)^r\le 1+rx$ for $x\ge -1$ and $ r\in[0,1]$ to $(1+(\hat{\varepsilon}^2n-1))^{4/A}$ we obtain
\[
  (1-2\hat{\varepsilon})^{-k}(1+2\hat{\varepsilon})-1
  \le 1+2\hat{\varepsilon}+4(1+2\hat{\varepsilon})\frac{\hat{\varepsilon}^2n-1}{A} - 1
  \le \frac{2\hat{\varepsilon}A+4\hat{\varepsilon}^2 n + 8\hat\varepsilon^3n}{A}.
\]
Then, since $k\ge 1$ we infer, say, $A \le \hat\varepsilon n$. 
Moreover, since $\hat\varepsilon \le 1/9$ we obtain $8\hat\varepsilon^3 n \le 8 \hat\varepsilon^2n$, so that
\[
  (1-2\hat{\varepsilon})^{-k}(1+2\hat{\varepsilon})-1
  < \frac{16 \hat{\varepsilon}^2 n}{A}, 
\]
Therefore, from~\eqref{LWBblue} we obtain $\mathbb{P}(\mathcal{T}(\varepsilon-\hat{\varepsilon}/2) > k) > A/(8\hat{\varepsilon}n)$ and so
\begin{equation*}\label{eq:branchingprocdeathblue}
    \mathbb{P}(\mathcal{T}(\varepsilon-\hat{\varepsilon}/2) \le k)^{2^{-10}\hat{\varepsilon}n}
    < \left(1-\frac{A}{8\hat{\varepsilon} n}\right)^{2^{-10}\hat{\varepsilon}n}
    \le e^{-A/2^{13}}.
\end{equation*}
This means that $\mathbb{P}(Z_k' = 0) \le e^{-A/2^{13}}$ in the right-hand side of~\eqref{eq:starblue}. To bound the other term in~\eqref{eq:starblue} we apply Lemma~\ref{lem:upperboundunhappysmalleps} with $i=1$, $j = k$, so that on ${\cal S}_1$
\[
    \mathbb{P}\big(\tau_1^U-\tau_1^L < k\mid \mathcal{F}_{\tau^L_1}\big)
    \le \exp\left(-\frac{\hat\varepsilon n}{2^{10}} \cdot \frac1k\right)
    \le \exp\left(-\frac{\hat\varepsilon n}{2^{10}} \cdot \frac{A}{\hat\varepsilon^{-1} \ln(\hat\varepsilon^2n)}\right).
\]
Note that since $x \mapsto x/\ln x$ is increasing for $x \ge e^2$ we have
$\hat\varepsilon^2 n / \ln(\hat\varepsilon^2 n) \ge e^2/2,$
completing the proof.
\end{proof}

Now we turn our attention to the case $\varepsilon>en^{-1/2}$, where $\varepsilon=\hat{\varepsilon}$ by the definition of $\hat{\varepsilon}$. 
As mentioned earlier, we can study the process using the lower coupling with the binomial process between $\tau_i^L$ and $\tau_i^U$ for $i\ge 1$. When $\varepsilon\le en^{-1/2}$ we only considered the case $i=1$ in the previous lemma, and this will be sufficient for the proof of Proposition~\ref{mainpropsec4}. 
However, if $\varepsilon > en^{-1/2}$, the branching processes associated to the binomial process are supercritical with survival probability $\Theta(\varepsilon)$, and the bounds that we obtain are weaker. In order to overcome this issue,  we will consider \emph{many}  of the intervals to obtain a stronger bound. In particular, let $k_0 \coloneqq  e^{2^{-16}\varepsilon^2n} $. Suppose that the number of unhappy particles drops below the threshold $\varepsilon n/32$; what is then the probability that we finish within the next $\varepsilon^{-1}k_0/A$ steps and \textit{never} go back above $\varepsilon n/8$ unhappy particles? The answer is provided by the following lemma.
\begin{lem}
\label{lem:localstop}
Let $k_0 \coloneqq e^{2^{-16}\varepsilon^2n}$, $b \coloneqq e^{2^{-19}\hat{\varepsilon}n}$ and $i\in\mathbb{N}$. Consider the event
\[
    \mathcal{S}_i\coloneqq\big\{2^{-10}\hat{\varepsilon}n \le U_{\tau_i^L}\le 2^{-5} \hat{\varepsilon}n, \tau_i^L \le b \big\}
    \in {\cal F}_{\tau_i^L}~.
\]
Then, for $\varepsilon>en^{-1/2}$ and $i\in\mathbb{N}$, on the event ${\cal S}_i$  we have
\begin{equation*}
    \mathbb{P}\big(T_{n,M} \le \tau_i^L + \varepsilon^{-1}k_0/A, \tau_i^U = \infty \mid \mathcal{F}_{\tau_i^L}\big)
    \le
    \left\{
	\begin{array}{ll}
		e^{-2^{-11}\varepsilon^2 n }  & \mbox{if } A \le k_0 \\
		e^{-2^{-11}A\varepsilon^2 n/k_0} & \mbox{if } A > k_0
	\end{array}
    \right.
    .
\end{equation*}
\end{lem}
\begin{proof}
Set $k\coloneqq \lfloor \varepsilon^{-1}k_0/A\rfloor$.
By applying Lemma~\ref{lem:lowercoupling} we infer, that as long as $\tau_i^L + t \le \tau_i^U$, the binomial process $Z_t(U_{\tau_i^L}, \varepsilon/2)$ provides a lower coupling for $U_{\tau_i^L + t}$. Thus, if we set $Z'_t \coloneqq Z_t( \lceil 2^{-10}\varepsilon n \rceil, \varepsilon/2)$, then on the event ${\cal S}_i$ the sought probability is at most
\begin{align*}
    \mathbb{P}\big(\tau_i^U = \infty, Z_{k}(U_{\tau_i^L}, \varepsilon/2)=0 \mid \mathcal{F}_{\tau_i^L} \big)
    \le \mathbb{P}\big(Z'_{k} = 0\big) .
\end{align*}
Fact~\ref{obs:binbranch} implies 
\begin{align*}
    \mathbb{P}\big(Z'_k=0\big)
    &\le
    \mathbb{P}\big(\mathcal{T}(\varepsilon/2)
    \leq k \big)^{2^{-10}\varepsilon n }.
\end{align*}
Let $\rho$ be the survival probability of $\mathcal{T}$, that is, the probability that $\mathcal{T}$ has infinitely many (non-empty) levels. Since $\mathcal{T}$ survives only if the root has two children and at least one of them survives, $\rho$ satisfies the equation
\[
    \rho = \frac12(1+\varepsilon/2)(2\rho - \rho^2),
\]
and so $\rho = \varepsilon/(1+\varepsilon/2)$. 
Now, if $A\le k_0$, then 
\begin{align*}
    \mathbb{P}\big(\mathcal{T}(\varepsilon/2)\le k \big)^{2^{-10}\varepsilon n }
        \leq \left(1-\frac{\varepsilon}{1+\varepsilon/2}\right)^{2^{-10}\varepsilon n}
        \leq e^{-2^{-11}\varepsilon^2 n }.
 \end{align*}
Next we consider the case $k_0<A\le \varepsilon^{-1} k_0$. By Lemma~\ref{BPgenrallemma1} and Bernoulli's inequality $(1+x)^{-N} \ge 1 - xN$ for $x > -1$ and $N \in \mathbb{N}$ and a little algebra  we have 
$$	   
    \mathbb{P}\big(\mathcal{T}(\varepsilon/2) \geq k+1\big)
    \ge \frac{\varepsilon/2}{1-(1-\varepsilon/2)(1+\varepsilon/2)^{-k}}
    \ge \frac1{k+1}.
$$
Since $A\le \varepsilon^{-1} k_0$ we infer $1 \le \varepsilon^{-1} k_0/A$. Together with $k \le  \varepsilon^{-1}k_0/A$ we conclude that
\[\mathbb{P}\big(\mathcal{T}(\varepsilon/2) \geq k+1\big) \ge {A\varepsilon}/(2k_0)\]
and therefore
 \begin{align*}
    \mathbb{P}\big(\mathcal{T}(\varepsilon/2)\le \lfloor \varepsilon^{-1}k_0/A \rfloor\big)^{2^{-10}\varepsilon n }
    \le \left(1-\frac{A\varepsilon}{2k_0}\right)^{2^{-10}\varepsilon n}
    \le e^{-2^{-11}A\varepsilon^2 n/k_0}.
 \end{align*}
Finally, for $\varepsilon^{-1}k_0<A$ we have
$\mathbb{P}(\mathcal{T}(\varepsilon/2)\leq \varepsilon^{-1}k_0/A) = 0$
and the result follows.
\end{proof}

We are now ready to establish the main result of the section, where the basic idea
is as follows. If $\varepsilon$ is 'small', namely $\varepsilon \le en^{-1/2}$, then  we use directly Lemma~\ref{lem:boundsmalleps} to obtain the desired bound. 
On the other hand, for all larger $\varepsilon$,
if we start the process with a number of unhappy particles below (but not too far from) $\varepsilon n/32$, then, either we will stop before ever reaching again $\varepsilon n/8$ unhappy particles, or we cross at least once that barrier. 
Lemma~\ref{lem:localstop} asserts that it is quite unlikely to stop within $\varepsilon^{-1}k_0/A$ steps without crossing the barrier. On the other hand, Lemma~\ref{lem:upperboundunhappysmalleps} asserts that it is extremely likely that the process will require at least $\varepsilon^{-1}/2$ steps before crossing the barrier. So, suppose that $\tau_i^U-\tau_i^L\ge \varepsilon^{-1}/2$ for every $i\ge 1$ for which $\tau_i^U<\infty$.
Clearly, if the process traverses
the strip $\{\varepsilon n/32,\dots,\varepsilon n/8\}$
more than $k$ times before stopping, we readily infer that $T_{n,M} \ge \varepsilon^{-1}k/2$. Therefore, if we consider the event $T_{n,M} \le \varepsilon^{-1}k/2$, this means (under our initial assumption) that we can have at most $k$ traversals and there must be a (last) time at which the number of unhappy particles goes below $\varepsilon n/32$ and reaches zero without ever returning above $\varepsilon n/8$. 
Putting everything together then yields the statement.
\paragraph{Proof of Proposition~\ref{mainpropsec4}.}
Set $b\coloneqq e^{2^{-19}\hat{\varepsilon}n}$ and, as before,
\begin{equation}\label{si}
    \mathcal{S}_i\coloneqq\big\{2^{-10}\hat{\varepsilon}n \le U_{\tau_i^L}\le 2^{-5} \hat{\varepsilon}n, \tau_i^L \le b\big\},
    \quad i\in \mathbb{N}.
\end{equation}
We first consider the case $\varepsilon\le e n^{-1/2}$. Set  $k\coloneqq \hat{\varepsilon}^{-1}\ln(\hat{\varepsilon}^2n)/A \ge 1$ (as otherwise there is nothing to show). Depending on whether ${\cal S}_1$ occurs or not and using the tower property of conditional expectation we obtain 
\begin{align*}
\mathbb{P}(T_{n,M} \le k)
    &\le \mathbb{P}\big(T_{n,M} \le \tau^L_1+k, {\cal S}_1\big)
        + \mathbb{P}(T_{n,M} \le k, \mathcal{S}_1^c) \\
    &\le \mathbb{P}\big(T_{n,M} \le \tau^L_1+k, {\cal S}_1\big)
        +\mathbb{P}\big(T_{n,M} \le k, \tau_1^L \ge b\big)
        + \mathbb{P}(\tau_1^L\leq b, \mathcal{S}_1^c) \\
    &= \mathbb{E}\Big[\mathbbm{1}_{\mathcal{S}_1}\mathbb{P}(T_{n,M} \le \tau_1^L + k\mid {\cal F}_{\tau^L_1})\Big]
    +\mathbb{P}(T_{n,M} \le k, \tau_1^L \ge b)
    +\mathbb{P}(\tau_1^L\leq b, \mathcal{S}_1^c).
\end{align*}
Using Lemma~\ref{lem:boundsmalleps} we see that the conditional probability is at most $2\exp(-A/2^{13}) $. Moreover, since $b> \hat\varepsilon^{-1}\ln(\hat\varepsilon^2n) \ge k$ for large $n$ (and uniformly in $\varepsilon$), the event $\{T_{n,M} \le k, \tau_1^L \ge b\}$ has probability zero for all large~$n$.
Finally, by Lemma \ref{boundbadevent2}, we have 
\begin{equation}\label{scomplement}
    \mathbb{P}(\tau_1^L\leq b, \mathcal{S}_1^c)
    \le \exp(-2^{-19}\hat\varepsilon n),
\end{equation}
which is actually true for all $\varepsilon$ such that $\varepsilon = o(1)$ and $|\varepsilon| < 1/9$ by the same lemma (we will use this later when we consider the case $\varepsilon > en^{-1/2}$). 
So, we obtain
\[
    \mathbb{P}(T_{n,M} \le k)
    \le
    2e^{-2^{-13}A}+e^{-2^{-19}\hat\varepsilon n},
\]
and the proof is completed in that case by noting that $A \le \hat\varepsilon n$.
In the rest of the proof we consider the case $\varepsilon > en^{-1/2}$, where $\hat\varepsilon = \varepsilon$. We set
\[
    k \coloneqq \varepsilon^{-1} k_0 / (2A),
    \quad
    \text{where}
    \quad
    k_0 = e^{2^{-16}\varepsilon^2n}~.
\]
We may assume that $k\ge 1$, as otherwise there is nothing to show.
Let us define $X$ to be the \textit{number of traversals} of the strip $\{\varepsilon n/32,\dots,\varepsilon n/8\}$ by the number of unhappy particles until dispersion, that is, $X \coloneqq \sup\{i \in \mathbb{N}_0 : \tau_i^U < \infty\}$.
Then
\begin{equation}\label{firstsplitlastpropblue}
    \mathbb{P}(T_{n,M} \le k)
    = \mathbb{P}(T_{n,M} \le k,X\geq k_0/A)
    +\mathbb{P}(T_{n,M} \le k,X< k_0/A).
\end{equation}
Suppose first that $A>k_0$ (and, since $k\ge1$, $A \le \varepsilon^{-1} k_0 / 2$). As $X$ is integer-valued
\begin{equation}\label{xiszero}
    \mathbb{P}(T_{n,M} \le k, X< k_0/A)
    =
    \mathbb{P}(T_{n,M} \le k, X=0).
\end{equation}
Note that  $k<b$ for large $n$ independently of $A$.
This fact, on the event $\{T_{n,M} \le k\}$, implies that the number of unhappy particles must drop below $\varepsilon n/32$ before time $b$.
So, using once more the tower property as before, the right-hand side in \eqref{xiszero} equals
\begin{align*}
    \mathbb{P}(T_{n,M} \le  k,X=0,\tau^L_1\leq b)
    & = \mathbb{P}(T_{n,M} \le  k,\tau^U_1=\infty,\tau^L_1\leq b)\\
    & \le \mathbb{P}\big(T_{n,M}\leq \tau^L_1+k,\tau_1^U=\infty,\tau^L_1\leq b\big)\\
    &\le \mathbb{E}\big[\mathbbm{1}_{ \mathcal{S}_1}\mathbb{P}\big(T_{n,M} \le \tau_1^L + k, \tau_1^U = \infty \mid \mathcal{F}_{\tau_1^L}\big)\big]+\mathbb{P}(\tau_1^L \le b, \mathcal{S}_1^c).
\end{align*}
Together with Lemma~\ref{lem:localstop} and \eqref{scomplement} this then implies (with $A \le \varepsilon^{-1}k_0 / 2$) 
\begin{equation}
\label{eq:TnMlekX=0}
    \mathbb{P}(T_{n,M} \le  k,X=0)
    \le e^{-2^{-11}A\varepsilon^2 n/k_0}+e^{-2^{-19}\varepsilon n}
    \le 2e^{-2^{-18}A\varepsilon^2 n/k_0}.
\end{equation}
To estimate the remaining probability term in \eqref{firstsplitlastpropblue} (i.e.\ the probability that $T_{n,M}\leq k$ and $X\geq k_0/A$) when $A > k_0$ we proceed as follows.
Denote by $\mathcal{M}_{i,j}$ the event that the number of unhappy particles remains below $\varepsilon n/8$ for at least~$j$ steps after $\tau_i^L$, and that $\tau_i^L \le b$. Then
\begin{align*}
    \mathbb{P}\left(T_{n,M} \le k,X\ge 1 \right)
    &\le \mathbb{P}\big(T_{n,M} \le k, X\ge 1,  \mathcal{M}_{1,k}\big) + \mathbb{P}\big(T_{n,M} \le k, \mathcal{M}_{1,k}^c\big) \\
    &\le \mathbb{P}\big(T_{n,M} \le k, \mathcal{M}_{1,k}^c \big),
\end{align*}
where the last inequality follows from the observation that, if the first traversal runs for at least $k$ steps, so does the whole process. 
Recall that $k<b$; then, by Lemma~\ref{lem:upperboundunhappysmalleps}, since $k \le \varepsilon^{-1}/2$ and $\tau^L_1 < \infty$ if $T_{n,M} \le k$, we obtain 
\[
    \mathbb{P}\big(T_{n,M } \le k, \mathcal{M}_{1,k}^c \big)
    \le \mathbb{E}\left[\mathbbm{1}_{\{\tau^L_1 < \infty\}}\mathbb{P}\big(\tau_1^U - \tau_1^L \le k \mid \mathcal{F}_{\tau^L_1}\big)\right]
    \le \exp\left(-\frac{A\varepsilon^2 n}{2^{11}k_0}\right).
\]
Substituting this last estimate back into~\eqref{firstsplitlastpropblue} and using~\eqref{eq:TnMlekX=0} settles the case $A > k_0$. 

In order to complete the proof, we now consider the remaining case $A \le k_0$ and $\varepsilon > e n^{-1/2}$.
Note that if the dispersion process terminates in at most $k$ steps, then $U_{\tau^L_{X+1}+t}=0$ for some  $t\le k$
and moreover, the process of unhappy particles does not go above $\varepsilon n/8$ after time $\tau_{X+1}^L$ before it terminates.
Recall also that we may assume $k < b$, so that $T_{n,M} \le k$ implies  $\tau_{X+1}^L \le b$. If in addition we recall the definition of $\mathcal{S}_i$ in (\ref{si}) and set
\[
    {\cal S} := \bigcap_{i \ge 1} {\cal S}_i,
    \quad
\]
then $\{T_{n,M} \le b\}\cap {\cal S}^c$ implies ${\cal A}(b)^c$ from Lemma~\ref{boundbadevent2}, so that $\mathbb{P}(\{T_{n,M} \le b\}\cap{\cal S}^c) \le \exp(-2^{-19}\varepsilon n)$. Therefore, going back to the second probability on the right-hand side of (\ref{firstsplitlastpropblue}), 
\begin{align*}
    \mathbb{P}\big(T_{n,M} \le k,&~X< k_0/A\big)
    \le \exp(-2^{-19}\varepsilon n) + \sum_{1 \le i < 1+k_0/A} \mathbb{P}(T_{n,M} \le k,X=i-1,{\cal S}_{i})\\
    &\le \exp(-2^{-19}\varepsilon n) + \sum_{1 \le i < 1+k_0/A} \mathbb{P}(\{T_{n,M} \le \tau_i^L + k, \tau_i^U = \infty\}\cap {\cal S}_{i})\\
    &= \exp(-2^{-19}\varepsilon n) + \sum_{1 \le i < 1+k_0/A} \mathbb{E}\big[\mathbbm{1}_{{\cal S}_i}\mathbb{P}(T_{n,M} \le \tau_i^L + k, \tau_i^U = \infty\mid\mathcal{F}_{\tau^L_{i}}) \big].
\end{align*}
By applying Lemma~\ref{lem:localstop} we obtain from $A \le k_0 =e^{2^{-16}\varepsilon^2n}$ and $|\varepsilon| < 1/9$
\begin{align}
\label{ffp}
    \mathbb{P}\big(T_{n,M} \le k,&X< k_0/A\big)
    \le \exp(-2^{-19}\varepsilon n) + \frac{2k_0}{A} \cdot \exp(-2^{-11}\varepsilon^2 n)
    \le \frac{3}{A}.
\end{align}
Suppose now that $T_{n,M} \le k$ and $X\ge  k_0/A$. Let again $\mathcal{M}_{i,j}$ be the event that the number of unhappy particles remains below $\varepsilon n/8$ for at least~$j$ steps after $\tau_i^L$, and that $\tau_i^L \le b$. Set
\[
    {\cal M} \coloneqq  \bigcap_{1 \le i < 1 + k_0/A} {\cal M}_{i, \varepsilon^{-1}/2}.
\]
By Lemma~\ref{lem:upperboundunhappysmalleps} and a union bound we see that
\[\mathbb{P}(\{T_{n,M} \le b\}\cap {\cal M}^c) \le \frac{2k_0}{A} \exp(-2^{-11}\varepsilon^2n).\]
Since $k<b$, the above estimate allows us to write
\begin{align*}
    \mathbb{P}(T_{n,M} \le k, X\ge k_0/A)&    \le \mathbb{P}\left(\big\{T_{n,M} \le k, X\ge k_0/A\big\}\cap \mathcal{M}\right)+\mathbb{P}(\{T_{n,M} \le b\}\cap {\cal M}^c)\\
    &\leq \mathbb{P}\left(\big\{T_{n,M} \le k, X\ge k_0/A\big\}\cap \mathcal{M}\right)+\frac{2k_0}{A} \exp(-2^{-11}\varepsilon^2n)\\
    &=\frac{2k_0}{A} \exp(-2^{-11}\varepsilon^2n),
\end{align*}
where the last equality follows as the event $\mathcal{M}\cap \{X\geq k_0/A\}$ implies that the dispersion process runs for more than $\varepsilon^{-1}k_0/(2A)$ steps, and thus the probability of the events occurring simultaneously is zero. 
In addition, $2k_0e^{-2^{-11}\varepsilon^2 n}/A\le 2/A$, whence we can conclude that
\[\mathbb{P}(T_{n,M}< k,X\geq k_0/A)\leq 2/A.\]
Together with~\eqref{ffp} this shows
$\mathbb{P}(T_{n,M} \le \varepsilon^{-1}k_0/(2A))\leq 5A^{-1}$ for $A\le k_0$,
completing the proof.
\qed

\bibliographystyle{plain}
\bibliography{ref}

\end{document}